\documentclass{amsart}
\usepackage[utf8]{inputenc}
\usepackage{amsfonts,latexsym,amssymb,amsmath,amsthm,color}
\usepackage{enumerate}
\usepackage{captdef}
\usepackage{enumerate,graphicx,graphpap}
\usepackage[matrix]{xy}
\usepackage{cases}
\usepackage{mathrsfs}
\usepackage{srcltx}
\usepackage[normalem]{ulem}
\input xy

\usepackage{color}

\definecolor{dgreen}{RGB}{68,156,3}

\newcounter{mnotecount}[section]

\newcommand{\rmnote}[1]{}

\newcommand{\ssso}[1]{}

\theoremstyle{plain}
\newtheorem{teor}{Theorem}[section]
\newtheorem{lema}[teor]{Lemma}
\newtheorem{coro}[teor]{Corollary}
\newtheorem{prop}[teor]{Proposition}

\theoremstyle{definition}
\newtheorem{defi}{Definition}[section]

\newtheorem{eje}{Example}[section]
\newtheorem{nota}[teor]{Remark}
\newtheorem*{gracias}{Acknowledgments}

\newtheoremstyle{teoremacita}
{3pt}
{3pt}
{\itshape}
{}
{\bfseries}
{}
{ }
{\thmname{#1}\thmnumber{ #2'}\thmnote{ #3}.}
\theoremstyle{teoremacita} \newtheorem*{teor*}{}

\newcommand{\be}{\begin{enumerate}}
\newcommand{\ee}{\end{enumerate}}

\newcommand{\bi}{\begin{itemize}}
\newcommand{\ei}{\end{itemize}}

\def\N{\mathbb N}

\def\Q{\mathbb Q}
\def\R{\mathbb R}
\def\C{\mathbb C}

\def\PC{\mathbb{P}_\C^1}
\def\OO{\mathcal{O}}

\def\a{\alpha}
\def\b{\beta}

\def\tt{{\boldsymbol t}}
\def\xx{{\boldsymbol x}}
\def\yy{{\boldsymbol y}}

\def\aa{{\boldsymbol \alpha}}
\def\bb{{\boldsymbol \beta}}
\def\gg{{\boldsymbol \gamma}}
\def\ee{{\boldsymbol \varepsilon}}
\def\00{{\boldsymbol 0}}
\def\11{{\boldsymbol 1}}

\def\RR{{\boldsymbol R}}
\def\PP{{P}}
\def\QQ{{Q}}

\def\d{\partial}

\begin{document}

\title[Tauberian theorems for $k$--summability with respect to an analytic germ]{Tauberian theorems for $k$--summability with respect to an analytic germ}

\author{Sergio A. Carrillo, Jorge Mozo-Fern\'{a}ndez, Reinhard  Sch\"{a}fke}
\address{(Sergio A. Carrillo) Fakult\"at f\"ur Mathematik, Universit\"at Wien, Oskar-Morgenstern-Platz~1, A-1090 Wien, Austria. Escuela de Ciencias Exactas e Ingenier\'{i}a, Universidad Sergio Arboleda, Calle 74, $\#$ 14-14, Bogot\'{a}, Colombia.}

\email{sergio.carrillo@univie.ac.at, sergio.carrillo@usa.edu.co}

\address{ (Jorge Mozo-Fern\'{a}ndez) Dpto. \'{A}lgebra, An\'{a}lisis Matem\'{a}tico, Geometr\'{i}a y Topolog\'{i}a, Facultad de Ciencias, Universidad de Valladolid, Campus Miguel Delibes, Paseo de Bel\'{e}n, 7, 47011 Valladolid, Spain.}

\email{jorge.mozo@uva.es}

\address{(Reinhard Sch\"{a}fke) Institut de Recherche Math\'{e}matique Avanc\'{e}e, U.F.R. de Math\'{e}matiques et Informatique, Universit\'{e} de Strasbourg et C.N.R.S., 7, rue Ren\'{e} Descartes, 67084 Strasbourg cedex, France.}

\email{schaefke@unistra.fr}

\thanks{The first author was supported by the Austrian FWF-Project P 26735-N25. The first and second authors are partially supported by the Ministerio de Econom\'{\i}a y Competitividad from Spain, under the Project ``\'{A}lgebra y geometr\'{\i}a en sistemas din\'{a}micos y foliaciones singulares" (Ref.: MTM2016-77642-C2-1-P)}


\date{\today}

\begin{abstract}The goal of this article is to establish tauberian theorems for the $k$--summability processes defined by germs of analytic functions in several complex variables. The proofs are based on the tauberian theorems for $k$--summability in one variable and in monomials, and a method of monomialization of germs of analytic functions. 
\end{abstract}

\maketitle

\section{Introduction}\label{Introduction}

This paper aims to be a step towards a general theory of multisummability in several variables, related with the notion of asymptotic expansion with respect to an analytic germ.

More precisely, consider a wide range of analytic problems, as differential, difference equations, or any other kind of more general functional equations. We can distinguish two situations when we search for local solutions at a point. At non-singular points it is customary to obtain analytic solutions under mild assumptions. It is at singular points where formal solutions occur. They can consist of formal power series including logarithms, exponential series or even more complicated objects. Then the question of how to associate  a true solution to a formal one is of great importance in the understanding of the given problem.

Borel summability and more generally $k$--summability are classical and have proved to be efficient summability methods in one variable to approach such problems, see \cite{Ramis1,Balser2}. For instance they have been successfully applied to holomorphic ordinary differential equations at irregular singular points and to families of partial differential equations in two variables of non-Kowalevskian type. Unfortunately these methods are not powerful enough to sum all formal power series solutions of the previous problems. A fundamental result in the theory of holomorphic ODEs at singular points is the \textit{multisummability} of its formal power series solutions: every formal solution can be built from $k$--summable series for different values of $k$. A cornerstone in the theory of multisummability is the following tauberian condition (see Theorem \ref{Tauberian classical}): a series $k$--summable for two different values of $k$ is convergent. 

For several variables the same questions appear when facing for instance singularly perturbed ordinary and partial differential equations. In previous works of the authors, $k$--summability in a monomial has been used effectively in these problems. In fact, the concept was investigated in detail for two variables in \cite{Monomial summ} and applied to doubly singular equations. 
Then it was also used in \cite{C, CM2} to sum formal power series solutions of singularly perturbed first order partial differential equations using a Borel-Laplace analysis adapted to this situation. 

Recently the notions of asymptotic expansions and $k$--summability in a germ of  analytic function have been defined and developed systematically in \cite{Sum wrt germs} by the second and third authors, generalizing their previous work \cite{Monomial summ} with M.\ Canalis-Durand for monomials. This theory has been proved  to behave well under blow-ups and it is also stable under the usual algebraic operations and differentiation. 
The authors are convinced that it is an important notion of summability in several variables useful
to sum formal solutions of certain partial differential equations and in the study of normal forms and reduction of singularities of holomorphic foliations. 

The tauberian theorems we present here contribute to the study of these new summability methods, extending naturally the ones for $k$--summability in one variable and the ones for monomials in \cite{CM} and \cite{C}. They provide a criterion to determine when two such methods are equivalent in the sense that they sum the same formal series. It turns out that they also associate the same value to a summable series in this case. Our theorems also imply that a series summable w.r.t.\ two essentially different methods is convergent and provide examples of series that cannot be summed w.r.t.\ any of them as it had been done in one dimension. The proofs are based on the  classical methods to prove tauberian properties for $k$--summability in one variable, c.f., \cite{RS89}, and on induction on the number of steps to monomialize the germs of analytic functions involved.

The structure of the paper is as follows: Section \ref{Remarks on Monomialization} contains the facts on monomialization of germs of analytic functions we will use during the paper, as presented in \cite{Sum wrt germs}. Section \ref{Asymptotic expansions and summability wrt an analytic germ} and \ref{Summability in an analytic germ} are devoted to recall the concept and main results on asymptotic expansions and $k$--summability in an analytic germ. Section \ref{Summability in an analytic germ} also includes new properties of $P$-$s$--Gevrey series. Finally Section \ref{Tauberian properties for summability} contains the main results of this work, namely, the tauberian properties for these summability methods (Theorems \ref{no sing directions then convergence for monomials} and \ref{Main Result}).

\begin{gracias}
{The first author wants to acknowledge professor Armin Rainer from University of Vienna for fruitful discussions and support under his FWF-Project P 26735-N25. The second author thanks University of Vienna for his stay there while preparing this work.}
\end{gracias}

\section{Remarks on Monomialization}\label{Remarks on Monomialization}

Let $\N$ denote the set of natural numbers including $0$ and $\N^+=\N\setminus\{0\}$.  We also write $\R{^+}$ for the set of positive real numbers.

Let $d\geq 2$ be an integer. We will work with $(\C^d,\00)$ and local coordinates $\xx=(x_1,\dots,x_d)$. We will write $\xx'=(x_2,\dots,x_d)$ and $\xx''=(x_3,\dots,x_d)$. $\widehat{\mathcal{O}}=\C[[\xx]]$ and  $\mathcal{O}=\C\{\xx\}$ will denote the rings of formal and convergent power series in $\xx$ with complex coefficients, respectively. $\widehat{\mathcal{O}}^\ast=\{U\in\widehat{\mathcal{O}}\hspace{0.1cm}|\hspace{0.1cm} U(\00)\neq 0\}$, $\mathcal{O}^\ast=\{U\in\mathcal{O} \hspace{0.1cm}|\hspace{0.1cm} U(\00)\neq 0\}$ will denote the corresponding groups of units. If $\boldsymbol{\beta}=(\beta_1,\dots,\beta_d)\in\N^d$ we use the multi-index notation $|\boldsymbol{\beta}|=\beta_1+\cdots+\beta_d$, $\bb!=\beta_1!\cdots\beta_d!$, $\boldsymbol{x}^{\boldsymbol{\beta}}=x_1^{\beta_1}\cdots x_d^{\beta_d}$ and  $\frac{\d^{\boldsymbol{\beta}}}{\d \boldsymbol{x}^{\boldsymbol{\beta}}}=\frac{\d^{|\boldsymbol{\beta}|}}{\d x_1^{\beta_1}\cdots \d x_d^{\beta_d}}$.

We will use blow-ups of codimension two smooth varieties. We choose the center of the blow-up to be $\{x_1=x_2=0\}$ and we will denote by  
$$M=\{([u_1,u_2],\tt)\in \mathbb{P}_\C^1\times \C^d | u_1t_2=u_2t_1\},\quad b:M\rightarrow \C^d,$$ the \textit{blow-up manifold} and the \textit{canonical projection} over the base space $\C^d$, respectively. Here $\mathbb{P}_\C^1$ denotes the complex projective line. The set $\mathbb{P}_\C^1\times \{(0,0)\}\times \C^{d-2}$ is called the \textit{exceptional divisor}. It is connected and for $d=2$ it is also compact.

$M$ is covered by affine charts, each one analytically equivalent to $\C^d$. In fact, identifying $\mathbb{P}_{\C}^1=\C\cup\{\infty\}$ as $[1,\xi]\equiv \xi\in\C$, $[0,1]\equiv\infty$, we use the charts centered at $\xi\in\C$ and $\infty$, \begin{align*}
\phi_\xi: M_\xi\longrightarrow \C^d,& \quad\left([u_1,u_2],\tt\right)\longmapsto \left(\frac{u_2}{u_1}-\xi,t_1,\tt''\right),\\
\phi_\infty: M_\infty\longrightarrow \C^d,&\quad \left([u_1,u_2],\tt\right)\longmapsto \left(\frac{u_1}{u_2},t_2,\tt''\right),
\end{align*} respectively, where $M_\xi=M_0=\{([u_1,u_2],\tt)\in M | u_1\neq 0\}$  and $M_\infty=\{([u_1,u_2],\tt)\in M | u_2\neq 0\}$. Then the map $b$ in the charts $\phi_\xi, \phi_\infty$ takes the form \begin{align*}
b_\xi=b\circ \phi_\xi^{-1}:\C^d\longrightarrow \C^d,& \quad\boldsymbol{v}\longmapsto (v_2,(\xi+v_1)v_2,\boldsymbol{v}''),\\
b_\infty=b\circ \phi_\infty^{-1}:\C^d\longrightarrow \C^d,&\quad \boldsymbol{v}\longmapsto (v_1v_2,v_2,\boldsymbol{v}'').
\end{align*} We will also use the ramifications determined by $$r_m:\C^d\longrightarrow\C^d,\quad \tt\longmapsto (t_1^m,\tt'), \quad m\geq2.$$

We say that $f\in\mathcal{O}$ has \textit{normal crossings} (at the origin) if there is a diffeomorphism $D\in\text{Diff}(\C^d,\00)$ such that $(f\circ D)(\xx)=\xx^{\bb}\cdot U(\xx)$ for some $\bb\in\N^d$ and $U\in\mathcal{O}^\ast$. Moreover, if $f_1\cdot f_2 \cdots f_n$ has normal crossings, then every $f_j$ has normal crossings. The converse is also true assuming that a common diffeomorphism $D$  can be found for all $f_j$.

Using these blow-ups and ramifications it is possible to achieve normal crossings for a given holomorphic map, see \cite{RSW}, and \cite[Lemma 2.1]{Sum wrt germs} for the main ideas of the proof.

\begin{lema}\label{Normalization} There exists a function $h:\mathcal{O}\setminus\{0\}\rightarrow \N$ with the following properties:
	\begin{enumerate}
		\item  If $h(f)=0$, then $f$ has normal crossings.
		
		\item If $h(f)>0$, then there exists a diffeomorphism $D\in\text{Diff}(\C^d,\00)$ such that either for all $\xi\in\mathbb{P}_\C^1$ $$h(f\circ D\circ b_\xi)<h(f),$$ or there exists $m\in\N$, $m\geq2$ such that $h(f\circ D\circ r_m)<h(f)$.
	\end{enumerate}
\end{lema}

Observe that we can assume that $h(f)\leq h(f\cdot g)$, for all $f,g\in\mathcal{O}\setminus\{0\}$, by redefining $h$ such that $h(f)=N$ is the minimal number of blow ups and ramifications in any chain of diffeomorphisms, blow ups and ramifications reducing $f$ to normal crossing. Indeed, since any factors of a germ having normal crossings must have normal crossings, too, any chain of diffeomorphisms, blow ups and ramifications reducing $f\cdot g$ to normal crossing
also reduces $f$.

We will need the following lemmas on convergence and associated elements under ramifications or blow-ups, see e.g., \cite[p. 493]{MM}. {The proof of Lemma \ref{Associated elements plus blow up} follows the same lines as at the end of the proof of Lemma 2.2 in \cite{Sum wrt germs}.}

\begin{lema}\label{Convergence is preserved by blow-ups}
	For any $\hat{f}\in\hat{\mathcal{O}}$ the following assertions are equivalent:\begin{enumerate}\item $\hat{f}\in \mathcal{O},$
		
		\item There exists $m\geq 2$ such that $\hat{f}\circ r_m\in \mathcal{O}$,
		
		\item There exists $\xi\in\mathbb{P}^1_{\C}$ such that $\hat{f}\circ b_\xi \in  \mathcal{O}$.
	\end{enumerate}
\end{lema}

\begin{proof}The only non-trivial statement is (3) implies (1). Consider $\hat{f}=\sum_{\bb\in\N^d} f_\bb \xx^\bb$ and assume $\hat{f}\circ b_\xi \in  \mathcal{O}$ for some  $\xi\in\mathbb{P}^1_{\C}$. Then there are coordinates $\yy=(y_1,y_2,\yy'')$ such that $\sum_{(\beta_1,\beta_2,\bb'')\in\N^d} f_\bb y_1^{\beta_1+\beta_2}y_2^{\beta_2}(\yy'')^{\bb''}$ is convergent, i.e., $|f_\bb|\leq CA^{|\bb|+\beta_2}$ for some constants $C,A>0$. Thus $\hat{f}$ is clearly convergent.
\end{proof}

\begin{lema}\label{Associated elements plus blow up} Consider $f_1, f_2\in\mathcal{O}$. Then $f_2=U\cdot f_1$ for some $U\in\mathcal{O}^\ast$ if and only if one (and in fact, both) of the following situations hold:\begin{enumerate}
		\item There exists $m\in\N$, $m\geq2$ and $U_m\in\mathcal{O}^\ast$ such that $f_2 \circ r_m= U_m \cdot (f_1 \circ r_m)$.
		
		\item For every $\xi\in\mathbb{P}^1_{\C}$ there exists $U_\xi\in\mathcal{O}^\ast$ such that $f_2 \circ b_\xi= U_\xi \cdot (f_1 \circ b_\xi)$.
	\end{enumerate}
\end{lema}

\begin{proof}Assume $f_1$ and $f_2$ are not identically zero. If (1) is true, then the function $U_m$ is invariant under right composition with the rotation $\boldsymbol{t}\mapsto (e^{2\pi i/m}t_1,\boldsymbol{t}')$, since the functions $f_j\circ r_m$, $j=1,2$ are invariant. Then $U_m=U\circ r_m$ for some $U\in\mathcal{O}^\ast$ as we wanted to show. 
	
	If (2) holds, then for every $\xi\in\PC$ there exists 
	an open neighborhood $\Omega_\xi'$ of $\00\in\C^d$, where $f_1\circ b_\xi$, $f_2\circ b_\xi$ and $U_\xi$ are defined. We consider the open neighborhoods $\Omega_\xi=\phi_\xi^{-1}(\Omega_\xi')$ of
	$(\xi,\00)\in M$ and the holomorphic function  $u_\xi:\Omega_\xi\to\C$ defined by $u_\xi=U_\xi\circ \phi_\xi$. 
	By definition, for $\xi,\zeta\in\PC$, we have
	$u_\xi (p)= u_\zeta (p)$ for $p\in \Omega_\xi\cap \Omega_\zeta$ with $f_1(b(p))\neq 0$.
	Since $f_1$ is not identically zero, this means that $u_\xi$ and $u_\zeta$ coincide on an open and dense subset
	of $\Omega_\xi\cap \Omega_\zeta$. Therefore $u_\xi=u_\zeta$ on this intersection and thus all
	$u_\xi$, $\xi\in\PC$, define a holomorphic function $u:\Omega\rightarrow \C$,
	where $\Omega$ is some neighborhood of $\PC\times \{ \00 \}\subseteq M$. 
	$\PC$ being compact, $u$ is constant over it, so there exists a holomorphic $U:\Omega'\rightarrow \C$,
	$\Omega'$ a neighborhood of $\00\in \C^d$, such that $U\circ b= u$ (apply Hartogs' Theorem).
	By construction, $U\in\OO^\ast$ and we have $f_2\circ b_\xi = (U\circ b_\xi)\cdot (f_1\circ b_\xi )$, for all $\xi\in\PC$. Thus we obtain that  $f_2= U\cdot f_1$, as desired.
\end{proof}

Finally, we will also make use of the monomial transformations $$\pi_{ij}:\C^d\rightarrow\C^d,\quad \pi_{ij}(\xx)=(x_1,\dots,\underbrace{x_ix_j}_{j \text{th entry}},\dots,x_d),\quad i,j=1,\dots,d, i\neq j,$$ that correspond to the charts of the blow-up with center $\{x_i=x_j=0\}$ and can be obtained from $b_0, b_\infty$ after permutations of the coordinates. We will call a \textit{monomial blow-up} a finite composition of these monomial transformations.

We will work with the partial order on $\R^d$ defined by $\boldsymbol{a}\leq \boldsymbol{b}$ if $a_j\leq b_j$, for all $j=1,\dots,d$. In particular $\boldsymbol{a}\not\leq \boldsymbol{b}$ if $a_j>b_j$ for some $j$. We will also write $\boldsymbol{a}< \boldsymbol{b}$ if $a_j<b_j$, for all $j=1,\dots,d$.

We consider $\Lambda_d:=(\N^d\setminus \{\00\})\times \R^+/\sim$ where $\sim$ is the equivalence relation defined by $(\aa,1/k)\sim (\aa',1/k')$ if $k\aa=k'\aa'$. {To simplify notations we identify an element $(\aa,1/k)$ and its equivalence class}. On $\Lambda_d$ we consider the partial order $\preceq$ given by $(\aa,1/k)\preceq(\aa',1/k')$ if $k\aa\leq  k'\aa'$. We will also write $(\aa,1/k)\prec(\aa',1/k')$ if  $k\aa < k'\aa'$.

The pull-back of $\xx^\aa$ under $\pi_{ij}$ is given by $\pi_{ij}^\ast(\xx^\aa)=\xx^{\aa}x_i^{\a_j}$, for any $\aa\in\N^d$. We will also denote by  $\pi_{ij}^\ast:\Lambda_d\rightarrow \Lambda_d$ the map given by $\pi_{ij}^\ast(\aa,1/k):=(\aa+\a_j e_i,1/k)$. Here $e_i$  is the $i$th vector of the canonical base of $\C^d$. We use analogous notations for any monomial blow-up $\pi:\C^{d}\rightarrow \C^{d}$.

\begin{nota}\label{Note on <}If the entries of $\aa'$ are nonzero, then $(\aa,1/k)\prec(\aa',1/k')$ if and only if $$\max_{1\leq j\leq d}\{\a_{j}/\a_{j}'\}<k'/k.$$
\end{nota}

We note $(\Lambda_d,\preceq)$ is not a totally-ordered set. However given a finite subset of it we can always apply an adequate monomial blow-up to obtain a totally-ordered set, as the following lemma shows.

\begin{lema}\label{Ordered monomials} Let $\{(\aa_i,1/k_i)\}_{i=1}^n\subset\Lambda_d$ be given. Then there is a monomial blow-up $\pi$ such that $\{\pi^\ast(\aa_i,1/k_i)\}_{i=1}^n$ is totally ordered w.r.t. $\prec$ and the entries of all the new monomials are different from zero.
\end{lema}

\begin{proof}It is straightforward to see that using the transformations $\pi_{ij}$ we may assume the entries of the monomials are not zero.
Then the problem is equivalent to order w.r.t. $<$ the vectors  $\{\bb_1,\dots,\bb_n\}\subseteq (\R^+)^d$, where $\bb_j=(\b_{j,1},\dots,\b_{j,d})=k_j\aa_j$.	

We say that $\bb_i$ and $\bb_j$ are comparable if $\bb_i\leq \bb_j$ or $\bb_j\leq \bb_i$. If $\bb_i$ and $\bb_j$ are not comparable, there exist indices $l\neq m$ such that $\b_{i,l}<\b_{j,l}$ and $\b_{i,m}>\b_{j,m}$. After the monomial blow-up $\pi_{ml}^{\circ N}$ with $N>\frac{\b_{i,m}-\b_{j,m}}{\b_{j,l}-\b_{i,l}},$ we obtain new vectors $\bb_i'$, $\bb_j'$ with $\b_{i,l}'=\b_{i,l}<\b_{j,l}=\b_{j,l}'$ and $\b_{i,m}'<\b_{j,m}'$. Then after a finite number of such transformations, we obtain comparable vectors. Note that if two vectors are comparable, further monomial blow-ups preserve this property. \end{proof}

\begin{nota}If all $k_i/k_j\in\Q$, i.e., $k_i/k_1=p_i/p_1$ for some integers $p_i$, take $\lambda=p_1/k_1=p_i/k_i$. Then $(\aa_i,1/k_i)\sim (p_i\aa_i, \lambda)$. In this case the problem is equivalent to monomialize the polynomial $$\prod_{i=1}^d \xx^{p_i\aa_i} \cdot \prod_{i<j} \left(\xx^{p_i\aa_i}-\xx^{p_j\aa_j}\right),$$ see e.g., \cite{BM}. Let us note that this is a particular case of a toric ideal (generated by products of binomials). In this case, reduction of singularities and monomialization turns out to be much simpler than in general cases. In particular, it is combinatorial, without the need of using diffeomorphisms during the process.
\end{nota}

We can extend the previous equivalence relation to germs of analytic functions other than monomials as follows. Consider $P_0,P_1\in\mathcal{O}\setminus\{0\}$, $P_0(\00)=P_1(\00)=0$ and $k_0, k_1>0$. We will write $(P_0,1/k_0)\sim (P_1,1/k_1)$ if we can find $p_0, p_1\in\N^+, U\in\mathcal{O}^\ast$ such that $$p_0/k_0=p_1/k_1,\quad \PP_0^{p_0}= U\cdot \PP_1^{p_1}.$$ This equivalence relation is preserved under ramifications and blow-ups as the following lemma shows.

\begin{lema}\label{Equivalence relation and blow ups}
	Consider $P_0,P_1\in\mathcal{O}\setminus\{0\}$, $P_0(\00)=P_1(\00)=0$ and $k_0, k_1>0$. Then $(P_0,1/k_0)\sim (P_1,1/k_1)$ if and only if one of the following statements hold:
	\begin{enumerate}
		\item There exists $m\in\N$, $m\geq2$ such that $(P_0\circ r_m,1/k_0)\sim (P_1\circ r_m,1/k_1)$.
		
		\item For every $\xi\in\mathbb{P}^1_{\C}$, $(P_0\circ b_\xi,1/k_0)\sim (P_1\circ b_\xi,1/k_1)$.
	\end{enumerate}
\end{lema}

\begin{proof}If $(P_0,1/k_0)\sim (P_1,1/k_1)$, both  conditions (1) and $(2)$ clearly hold. Conversely, if (1) holds, we can find $p_{0}, p_{1}\in\N^+$ and $U_m\in\mathcal{O}^\ast$, such that $$p_{0}/k_0=p_{1}/k_1,\quad (P_0\circ r_m)^{p_{0}}=U_m\cdot  (\PP_1\circ r_m)^{p_{1}}.$$ Then the relation $(P_0,1/k_0)\sim (P_1,1/k_1)$ follows from Lemma \ref{Associated elements plus blow up}. Now assume that (2) holds. By definition we can find $p_{0,\xi}, p_{1,\xi}\in\N^+$ and $U_{\xi}\in\mathcal{O}^\ast$, such that $$p_{0,\xi}/k_0=p_{1,\xi}/k_1,\quad (P_0\circ b_\xi)^{p_{0,\xi}}=U_{\xi}\cdot  (\PP_1\circ b_\xi)^{p_{1,\xi}}.$$ Now write  $k_0/k_1=p_0/p_1$ where $p_0,p_1\in\N^+$ and $(p_0,p_1)=1$. Then for each $\xi\in\mathbb{P}_\C^1$ we can find $m_\xi\in \N^+$ such that $p_{0,\xi}=m_\xi p_0$ and $p_{1,\xi}=m_\xi p_1$. We conclude that $(\PP_0\circ b_\xi)^{p_0}/(\PP_1\circ b_\xi)^{p_1}\in \mathcal{O}^\ast$ for each $\xi\in\mathbb{P}_\C^1$. Then again by Lemma \ref{Associated elements plus blow up} there is $U\in \mathcal{O}^\ast$ such that $\PP_0^{p_0}=U\cdot \PP_1^{p_1}$, i.e., $(P_0,1/k_0)\sim (P_1,1/k_1)$.
\end{proof}

\section{Asymptotic expansions in an analytic germ}\label{Asymptotic expansions and summability wrt an analytic germ}

The goal of this section is to present a summary on asymptotic expansion in a germ of analytic function, based on \cite{Sum wrt germs}. We have not included proofs, except for Lemma \ref{Vandermonde}. The aim here is only to establish the necessary background to be able to state and prove the tauberian properties in Section \ref{Tauberian properties for summability}.

Usual domains in $\C$ where holomorphic maps admit an asymptotic expansion are sectors at e.g., the origin. We will denote them as $$V(a,b;r):=\{t\in\C \hspace{0.1cm}|\hspace{0.1cm} 0<|t|<r, a<\text{arg}(t)<b\}=S(\theta,b-a;r),$$ emphasizing on its \textit{bisecting direction} $\theta=(b+a)/2$,  \textit{opening} $b-a$ and radius $0<r\leq\infty$. A subsector of $S$ is simply $S'=S(\theta',b'-a';r')$ where $a<a'<b'<b$, $0<r'<r$. The disk centered at the origin with radius $r>0$ will be denoted by $D_r=\{t\in\C \hspace{0.1cm}|\hspace{0.1cm} |t|<r\}$.

We fix a Banach space $(E,\|\cdot \|)$. We will use the notation $\mathcal{O}(\Omega, E)$ and $\mathcal{O}_b(\Omega, E)$ for the space of holomorphic and holomorphic and bounded $E$-valued maps defined on an open set $\Omega\subseteq \C^d$.

Let $S$ be a sector, $f\in\mathcal{O}(S, E)$ and let $\hat{f}=\sum_{n=0}^{\infty} a_n t^n\in E[[t]]$ be its \textit{asymptotic expansion} on $S$ (written $f\sim \hat{f}$ as $S\ni t\rightarrow 0$), i.e.,  for each subsector $S'$ and each $N\in\N$, there exists $C_N(S')>0$ such that \begin{equation}\label{def asym classic}
\left\|f(t)-\sum_{n=0}^{N-1} a_n t^n\right\|\leq C_N(S')|t|^N, \hspace{0.3cm} \text{ on } S'.
\end{equation}

The asymptotic expansion is said to be of \textit{$s$--Gevrey type} ($s>0$ and written $f\sim_s \hat{f}$ as $S\ni t\rightarrow 0$) if we can choose $C_N(S')=C(S')A(S')^N N!^s$, for some $C(S'), A(S')$ independent of $N$. In this case we conclude that $\hat{f}\in E[[t]]_s$ is a $s$-\textit{Gevrey series in $t$}, i.e., there are constants $B,D>0$ such that $\|a_n\|\leq BD^n n!^s$, for all $n\in\N$.

\begin{nota}\label{Lema polinomios} To have $f\sim_s\hat{f}$ as $S\ni t\rightarrow 0$ it is actually sufficient to have inequalities (\ref{def asym classic}) only for the values $N=Mp$, where $M\in \N^+$ is fixed. The reader may check this assertion with the aid of the following lemma that we will use later.
\end{nota}

\begin{lema}\label{Vandermonde} Let $V=V(a,b;r)$ be a sector, $0<\rho<r$ and $M$ be a positive integer. There is a constant $C_V(\rho,M)$ with the following property:
If $H(t)=a_0+a_1t+\cdots+a_Mt^{M-1}\in E[t]$ is a polynomial and $K:V\rightarrow \R^+$ is a map such that $\left\|H(t)\right\|\leq K(|t|)$, for all $t\in V$, then $\|a_j\|\leq C_V(\rho,M) K(\rho)$, for all $j.$
\end{lema}


\begin{proof}Take $t_0, t_1,\dots,t_{M-1}$ distinct points in $V$ with $|t_j|=\rho$  and let $G=(t_i^j)_{0\leq i,j\leq M-1}$ the corresponding Vandermonde matrix. On $E^M$ consider the norm $\|(z_1,\dots,z_{M})^t\|_1:=\sum_{j=1}^{M} \|z_j\|$ and on $\C^{M\times M}$ the corresponding matrix norm $\|A\|_1=\sup_{v\in E^M\setminus\{0\}} \|Av\|_1/\|v\|_1$ . Then from our hypothesis we see that $$\|a_{j}\|\leq \|(a_{0},\dots,a_{M-1})^t\|_1=\|G^{-1}G(a_{0},\dots,a_{M-1})^t\|_1\leq\|G^{-1}\|_1\sum_{i=1}^M\|H(t_i)\|\leq  M\|G^{-1}\|_1 K(\rho),$$ as required.\end{proof}

A key point to generalize asymptotic expansions in a germ is the following: $f\sim \hat{f}$ as $S\ni t\rightarrow 0$ if and only if  there exists $(f_N)_{N\in\N}\subset\mathcal{O}_b(D_R, E) $ such that for all subsectors $S'$ and $N\in\N$ there are constants $C_N(S')>0$ such that \begin{equation*}\label{def asym classic with sequence of function}\|f(t)-f_N(t)\|\leq C_N(S')|t|^N, \hspace{0.3cm}\text{ on }S'\cap D_R.\end{equation*} In the case $f\sim_s\hat{f}$ as $S\ni t\rightarrow 0$ we also require that $C_N(S')=C(S')A(S')^N N!^s$ and $\|f_N(t)\|\leq DB^N N!^s$, for all $|t|\leq R$ and $N\in\N$, for some constants $C(S'),A(S'),B,D>0$. In any case, the series $\hat{f}$ is completely determined by $f$ since $a_n=\lim_{S'\ni t\rightarrow 0} \frac{f^{(n)}(t)}{n!}$, for any subsector $S'$, and it is given by the limit of the Taylor series at the origin of the $f_n$, in the $\mathfrak{m}$-topology of $E[[t]]$, $\mathfrak{m}=(t)$.

For several variables, we use the notation $\hat{\mathcal{O}}(E)=E[[\xx]]$  and $\mathcal{O}(E)=E\{\xx\}$  for the space of formal and convergent power series in $\xx$ with coefficients in $E$, respectively. For any $\boldsymbol{r}=(r_1,\dots,r_d)\in(\R^+)^d$, $D_{\boldsymbol{r}}=\{\xx\in \C^d \hspace{0.1cm}|\hspace{0.1cm} |x_j|<r_j, j=1,\dots,d\}$ will denote the polydisk centered at the origin with polyradius $\boldsymbol{r}$. If $r_j=r$, for all $j$, we will write the Cartesian product $D_r^d$ instead. We denote by $J:\mathcal{O}(D_{\boldsymbol{r}},E)\rightarrow \mathcal{O}(E)$ the Taylor map assigning to a function its Taylor series at the origin.

Given $\boldsymbol{\a}\in\N^{d}\setminus \{\00\}$, any 	power series  $\hat{f}=\sum_{\bb\in\N^d} f_\bb \xx^\bb\in \hat{\mathcal{O}}(E)$ can be written uniquely as \begin{equation}\label{para definir Tpq}
\hat{f}=\sum_{n=0}^\infty \hat{f}_{\aa,n}(\boldsymbol{x})\boldsymbol{x}^{n\boldsymbol{\a}},\quad \hat{f}_{\aa,n}(\xx)=\sum_{\aa\not\leq\bb}f_{n\boldsymbol{\a}+\boldsymbol{\beta}}\boldsymbol{x}^{\boldsymbol{\beta}}.
\end{equation} 

Analogously, given $\PP \in \hat{\mathcal{O}}\setminus\{0\}$, $\PP(\00)= 0$ and an injective linear form $\ell:\N^d\rightarrow\R^+$,  $\ell(\aa)=\ell_1\a_1+\cdots+\ell_d\a_d$, every $\hat{f}\in \hat{\mathcal{O}}(E)$ can be written uniquely in the form \begin{equation}\label{to define T_l}
\hat{f}=\sum_{n=0}^\infty \hat{f}_{\PP,\ell,n}(\xx) \PP^n, \quad \hat{f}_{\PP,\ell,n}(\xx)\in \Delta_{\ell}(\PP,E).	
\end{equation}

Here the linear form $\ell$ is used to order the monomials by:   $\xx^\aa<_\ell\xx^\bb\text{ if }\ell(\aa)<\ell(\bb)$. We write  $\nu_\ell(\hat{f})=\aa$ if $\xx^\aa=\min_\ell\{\xx^\bb | f_\bb\neq 0\}$, where the minimum is taken according to $<_\ell$ and $$\Delta_\ell(\PP,E):=\left\{\sum g_\bb \xx^\bb\in \hat{\mathcal{O}}(E) \hspace{0.1cm}|\hspace{0.1cm} g_\bb=0 \text{ if } \bb\in \nu_\ell(\PP)+\N^d\right\}.$$ In the case $\PP=\xx^\aa$ we will simply write $\Delta(\xx^\aa,E)$. Then $\Delta_\ell(\PP,E)=\Delta(\xx^{\nu_\ell(\PP)},E)$. The decomposition (\ref{to define T_l}) follows from the \textit{Generalized Weierstrass Division} determined by $\PP$ and $\ell$, see \cite[Lemma 2.4, 2.6]{Sum wrt germs}.

\begin{prop}\label{Generalized Weierstrass Division} Let $P$ and $\ell$ as above. For every $\hat{g}\in \hat{\mathcal{O}}(E)$, there exist unique $q\in \hat{\mathcal{O}}(E)$ and $r\in \Delta_\ell(\PP,E)$ such that $g=q\PP+r$. Furthermore, if $\rho>0$ is  sufficiently small, then for every $g\in\mathcal{O}_b(D_{\rho(\ell)})$, $\rho(\ell)=(\rho^{\ell_1},\dots,\rho^{\ell_d})$, there exist unique $r\in \mathcal{O}_b(D_{\rho(\ell)})$ with $J(r)\in \Delta_\ell(P,E)$ and $q\in \mathcal{O}_b(D_{\rho(\ell)})$ such that $g=qP+r$. The corresponding operators $$Q_{P,\ell}, R_{P,\ell}:\mathcal{O}_b(D_{\rho(\ell)})\rightarrow \mathcal{O}_b(D_{\rho(\ell)}),\quad g\mapsto Q_{P,\ell}(g)=q, g\mapsto R_{P,\ell}(g)=r,$$ are linear and continuous.
\end{prop}



\begin{nota}\label{Remark 1}We remark the following facts that will be used later:
\begin{enumerate}\item For any $\hat{f}\in \hat{\mathcal{O}}(E)$ and $N\in\N^+$, since decomposition (\ref{to define T_l}) is unique, we have the relation  $$\hat{f}_{\PP^N,\ell,n}=\hat{f}_{\PP,\ell,nN}+\hat{f}_{\PP,\ell,nN+1}\PP+\cdots+\hat{f}_{\PP,\ell,nN+N-1}\PP^{N-1}.$$ 
		
\item If $\PP\in \mathcal{O}\setminus\{0\}$, then given $f\in \mathcal{O}(E)$ there exist $\rho>0$ and a unique sequence $(f_{\PP,\ell,n})_{n\in\N}$ in $\mathcal{O}_b(D_\rho^d,E)$ with $J(f_{\PP,\ell,n})\in \Delta_\ell(\PP,E)$, for all $n\in\N$, such that $f$ can also be written in the
		form $$f(\xx)=\sum_{n=0}^\infty f_{\PP,\ell,n}(\xx) \PP(\xx)^n,\quad \text{ for } |\xx|:=\max_{1\leq j\leq d} |x_j|\leq \rho.$$ Using the operators $Q_{P,\ell}$, $R_{P,\ell}$, the functions $f_{P,\ell,n}$ are given by $$f_{P,\ell,n}=R_{P,\ell}\circ Q_{P,\ell}^n(f).$$
\end{enumerate}
\end{nota}

Using decompositions (\ref{para definir Tpq}) and (\ref{to define T_l}) we obtain isomorphisms \begin{equation}\label{Isomoporphisms Ta TP}
	\hat{T}_{\boldsymbol{\a}}:\hat{\mathcal{O}}(E)\rightarrow \Delta(\xx^\aa,E)[[t]],\quad \hat{T}_{\PP,\ell}:\hat{\mathcal{O}}(E)\rightarrow \Delta_\ell(\PP,E)[[t]],
\end{equation} that satisfy $(\hat{T}_\aa\hat{f})(\xx^\aa)=\hat{f}$ and $(\hat{T}_{\PP,\ell} \hat{f})(\PP)=\hat{f}$, for all series $\hat{f}\in \hat{\mathcal{O}}(E)$, i.e., when we substitute $t=\xx^\aa$ or $t=\PP(\xx)$, respectively, we recover the initial series $\hat{f}$.

From now on, we will assume $\PP\in \mathcal{O}\setminus\{0\}$, $\PP(\00)=0$, is a germ of analytic function. For asymptotic expansions in $\xx^\aa$ or $\PP$ we will need to restrict our attention to formal power series in $\hat{\mathcal{O}}(E)$ for which the application of the previous isomorphisms gives us meaningful coefficients, i.e.,\ holomorphic maps. For this purpose we introduce the following spaces:

\begin{enumerate}
\item For the monomial case we will denote: \begin{align*}
\hat{\mathcal{O}}'_{r}(E)&:=\bigcap_{j=1}^d I_j(\mathcal{O}_b(D_r^{d-1},E)[[x_j]]),&  \mathcal{E}^{\boldsymbol{\a}}_r&:=\{ g\in\mathcal{O}_b(D_r^d,E)\hspace{0.1cm}|\hspace{0.1cm} J(g)\in \Delta(\xx^\aa,E)\},\\
\hat{\mathcal{O}}'(E)&:=\bigcup_{r>0} \hat{\mathcal{O}}'_r(E),& \mathcal{E}^{\boldsymbol{\a}}&:=\bigcup_{r>0}\mathcal{E}^{\boldsymbol{\a}}_r.
\end{align*} 
Here $I_j(\sum_n f_nx_j^n)=\sum_n I_j(f_n)x_j^n$, $I_j(f)(\xx)=f(x_1,\dots,x_{j-1},x_{j+1},\dots,x_d)$ and $\hat{\mathcal{O}}'_{r}(E)$ is embedded into $\hat{\mathcal{O}}'_{r'}(E)$, $0<r'<r$, by restriction. If $\hat{f}\in \hat{\mathcal{O}}'(E)$, then $\hat{T}_\aa \hat{f}\in \mathcal{E}^\aa[[t]] $ and all coefficients $\hat{f}_{\aa,n}=f_{\aa,n}$ have a common radius of convergence at the origin. 

\

\item For the general case we recall that a sequence $\{f_n\}_{n\in\N}\subset \mathcal{O}_b(D_r^d,E)$ is an \textit{asymptotic sequence for} $\hat{f}\in \hat{\mathcal{O}}(E)$ if $J(f_n)$ converges in the $\mathfrak{m}$-adic topology of $\hat{\mathcal{O}}(E)$ ($\mathfrak{m}=(\xx)$) to $\hat{f}$. If, moreover, $J(f_n)\equiv \hat{f} \text{ mod } \PP^n \hat{\mathcal{O}}(E)$, for all $n\in\N$, then we will say that the sequence is a \textit{$\PP$--asymptotic sequence for $\hat{f}$}. Then we can define:

\begin{align*}
\hat{\mathcal{O}}_{r}^\PP(E)&:=\left\{\hat{f}\in \hat{\mathcal{O}}(E) \hspace{0.1cm}|\hspace{0.1cm} \hat{f} \text{ has a }\PP-\text{asymptotic }\right.& 
\mathcal{E}^\PP_{\ell,r}&:=\mathcal{E}^{\nu_\ell(\PP)}_r,\\
&\hspace{2,9cm}\left.\text{sequence in } \mathcal{O}_b(D_r^d,E) \right\}, & &\\
\hat{\mathcal{O}}^\PP(E)&:=\bigcup_{r>0} \hat{\mathcal{O}}^\PP_{r}(E),&\mathcal{E}^{\PP}_{\ell}&:=\bigcup_{r>0}\mathcal{E}^{\PP}_{\ell,r}.
\end{align*}
We will refer to the elements of $\hat{\mathcal{O}}^\PP(E)$ as \textit{$\PP$--asymptotic series}. Note that $\mathcal{E}^{\boldsymbol{\a}}_r$ and $\mathcal{E}^\PP_{\ell,r}$ become Banach spaces with the norm $\|g\|=\sup_{|\xx|< r} \|g(\xx)\|$.
\end{enumerate}

If $\hat{f}\in \hat{\mathcal{O}}^\PP(E)$, then $\hat{T}_{\PP,\ell}\hat{f}\in \mathcal{E}^\PP_{\ell}[[t]]$ and all coefficients $\hat{f}_{\PP,\ell,n}=f_{\PP,\ell,n}$ have a common radius of convergence at the origin \cite[Coro. 4.10]{Sum wrt germs}. When necessary we will employ the notation $\hat{T}_{\PP,\ell}\hat{f}\mid_{D_\rho^d}$ to empathize the fact that the coefficients $f_{\PP,\ell,n}$ are defined on $D_\rho^d$, i.e.,  $\hat{T}_{\PP,\ell}\hat{f}\mid_{D_\rho^d}=\sum_{n=0}^\infty f_{\PP,\ell,n}  t^n\in \mathcal{E}^\PP_{\ell,\rho}[[t]]$.

In the analytic setting it is natural to work with \textit{$\PP$-sectors}, i.e., sets of the form $$\Pi_{P}=\Pi_{P}(a,b;\RR)=\left\{\boldsymbol{x}\in \C^d \hspace{0.1cm}|\hspace{0.1cm} \PP(\xx)\neq0,  a<\text{arg}(\PP(\xx))<b,\hspace{0.1cm} 0<|x_j|<R_j, \text{ for }j=1, \dots,d\right\},$$ where $a<b$ are real numbers and $\RR=(R_1,\dots,R_d)\in(\R^+)^d$ is a polyradius. For $\PP(\xx)=\xx^\aa$ we will simply write $\Pi_\aa$. The values $b-a$ and $\theta=(b+a)/2$ are called the \textit{opening} and the \textit{bisecting direction} of the $\PP$-sector $\Pi_\PP$. We will also use the notation $\Pi_{\PP}(a,b;\RR)=S_{\PP}(\theta,b-a;\RR)=S_{\PP}$. The notion of subsector is also clear.

Here any convenient branch of $\arg$ may be used. Anyhow, we we will only consider $\PP$-sectors of opening not greater than $2\pi$.

It is possible to construct operators $T_\aa$ and $T_{\PP,\ell}$ sharing the same properties as their formal counterparts (\ref{Isomoporphisms Ta TP}) for holomorphic maps defined on $\xx^\aa$- and $\PP$-sectors, respectively. We recall this main and technical result \cite[Lemma 3.8, Thm. 4.7]{Sum wrt germs} in the following theorem.

\begin{teor}\label{Existence TlP} Let $\ell:\N^d\rightarrow \R^+$ be an injective linear form and $\PP\in \mathcal{O}\setminus\{0\}$, $\PP(\00)=0$. Let $\Pi_\PP=\Pi_\PP(a,b;\RR)$ be a $\PP$-sector. Then there exists $\rho, \sigma ,L>0$ such that $\PP(D_{\rho}^d)\subset D_\sigma$ and the following properties hold:
\begin{enumerate}
\item If $f:\Pi_\PP\rightarrow E$ is a holomorphic map, then there exists a uniquely determined holomorphic map $T_{\PP,\ell}f:V(a, b;\sigma)\times D_\rho^d\rightarrow E$ such that $J((T_{\PP,\ell}f)(t,\cdot))\in \Delta_\ell(\PP,E)$ for any $t$ and $$(T_{\PP,\ell}f)(\PP(\xx),\xx)=f(\xx),\quad \xx\in \Pi_\PP, |\xx|<\rho .$$

\item Given a function $K:(0,S)\rightarrow \R^+$, $S\geq \sup_{\xx\in	\Pi_\PP} |\PP(\xx)|$, such that $\|f(\xx)\|\leq K(|\PP(\xx)|)$ for $\xx\in \Pi_\PP$, we have $$\|(T_{\PP,\ell} f)(t,\xx)\|\leq \frac{L}{|t|} K(|t|),\quad t\in V(a,b;\sigma), |\xx|< \rho.$$

\item If $\PP(\xx)=\xx^\aa$, then we can choose $\sigma=\RR^\aa$ and the inequality in (2) takes the form $$\|(T_{\aa} f)(t,\xx)\|\leq \frac{\RR^\aa}{|t|} K(|t|)\prod_{j=1}^d \left(1-\frac{|x_j|}{R_j}\right)^{-1},\quad t\in V(a,b;\RR^\aa), \xx\in D_\RR.$$
\end{enumerate}
\end{teor}

Finally we are in position to recall the notion of $\xx^\aa$- and $\PP$--asymptotic expansions.

\begin{defi}\label{Def P-asym}Let $f\in\mathcal{O}(\Pi_{\boldsymbol{\PP}},E)$, $\Pi_{\PP}=\Pi_{\PP}(a,b;\RR)$ and $\hat{f}\in \hat{\mathcal{O}}(E)$. We will say that $f$ has $\hat{f}$ as \textit{$\PP$--asymptotic expansion on $\Pi_{\PP}$} if $\hat{f}\in \hat{\mathcal{O}}^\PP_{r}(E)$ for some $r>0$ and if there is a $\PP$--asymptotic sequence $\{f_n\}_{n\in\N}$ in $ \mathcal{O}_b(D_r^d,E)$ such that for all $N\in\N$ and every subsector $\Pi_\PP'\subset \Pi_\PP$ there exists $C_N>0$ such that \begin{equation}\label{Def P-asym inq}
\|f(\xx)-f_N(\xx)\|\leq C_N |\PP(\xx)|^N,\quad \text{ on } \Pi'_\PP\cap D_r^d.
\end{equation}
We will denote this situation by $f\sim^{\PP} \hat{f}$ on $\Pi_{\PP}$. If $\PP(\xx)=\xx^\aa$,  we will write $f\sim^{\aa} \hat{f}$ on $\Pi_{\aa}$.
\end{defi}

The main purpose of the operators $\hat{T}_{\aa}, T_{\aa}, \hat{T}_{\PP,\ell}$ and $T_{\PP,\ell}$ is to provide a characterization of $\xx^\aa$- and $\PP$--asymptotic expansion in terms of classical asymptotic expansions in one variable, respectively. In this context, we state the following result \cite[Thm. 4.9]{Sum wrt germs}.

\begin{teor}\label{C P-asym in terms of Tl}
	Let $\ell:\N^d\rightarrow \R^+$ be an injective linear form. Then $f\sim^\PP \hat{f}$ on $\Pi_\PP=\Pi_\PP(a,b;\RR)$ if and only if there exists $\rho>0$ such that $\hat{T}_{\PP,\ell}\hat{f}\mid_{D_\rho^d}=\sum_{n=0}^\infty f_{\PP,\ell,n}  t^n\in \mathcal{E}^\PP_{\ell,\rho}[[t]]$ and one of the following two equivalent conditions holds:\begin{enumerate}
		\item We can choose $f_N=\sum_{n=0}^{N-1} f_{\PP,\ell,n}\PP^n$ in Definition \ref{Def P-asym}, i.e., for every $N\in\N$ and $\Pi'_\PP\subset \Pi_\PP$ there exists $L_N>0$ such that \begin{equation}\label{Formula Pasym}
		\left\|f(\boldsymbol{x})-\sum_{n=0}^{N-1}f_{\PP,\ell,n}(\xx)\PP(\xx)^n \right\|\leq L_N|\PP(\xx)|^N,\quad \text{ on } \Pi_{\PP}'\cap D_\rho^d.
		\end{equation}
		
		\item The function $T_{\PP,\ell}f$ from Theorem \ref{Existence TlP} is defined on $V(a,b;\sigma)\times D_\rho^d\rightarrow\C$ for some positive $\sigma$ and satisfies $$T_{\PP,\ell} f\sim \hat{T}_{\PP,\ell}\hat{f}\mid_{D_\rho^d} \text{ as } V (a,b;\sigma)\ni t\rightarrow 0.$$
	\end{enumerate}
\end{teor}

\begin{nota}\label{Remark 2}We remark the following facts on the notion of $\PP$--asymptotic expansions:
\begin{enumerate}\item The previous definition is independent of the chosen $\PP$--asymptotic sequence with limit $\hat{f}$.

\item $\PP$--asymptotic expansions are stable under addition and partial derivatives. If $E$ is a Banach algebra, then $\hat{\mathcal{O}}(E)$ is an algebra. In this case $\PP$--asymptotic expansions are stable under products as well. This is not obvious from the definition, except for addition.

\item The $\PP$--asymptotic expansion of a function on a $\PP$-sector, if it exists, is unique. Indeed, if $f\sim^\PP \hat{f}=\sum f_\bb \xx^\bb$ on $\Pi_\PP$,  then $$\lim_{\Pi'_{\PP}\ni\boldsymbol{x}\rightarrow\boldsymbol{0}} \frac{1}{\boldsymbol{\beta}!}\frac{\d^{\boldsymbol{\beta}}f}{\d \boldsymbol{x}^{\boldsymbol{\beta}}}(\boldsymbol{x})=f_\bb,\quad \text{for any subsector } \Pi_\PP'\subset \Pi_\PP.$$ For $\bb=\00$ the formula follows from inequality (\ref{Formula Pasym}) for $N=1$ in Theorem \ref{C P-asym in terms of Tl}(2). For an arbitrary $\bb$ the limit follows using the stability of $\sim^\PP$ under derivatives.

\item Consider two associated elements $\PP, \QQ\in \mathcal{O}\setminus\{0\}$, i.e.,\ $\QQ=U\cdot \PP$ where $U\in\mathcal{O}^\ast$ is a unit. Then  $\hat{\mathcal{O}}^\PP(E)=\hat{\mathcal{O}}^{\QQ}(E)$. Furthermore, if $|\xx|<r$, then $\theta_1<\text{arg } U(\xx)<\theta_2$, for some $\theta_1<\theta_2$ and $\theta_2-\theta_1$ can be made as small as desired if $r$ is small enough. It follows that if $f\sim^\PP \hat{f}$ on $\Pi_\PP(a,b;\RR)$, then $f\sim^{\QQ} \hat{f}$ on $\Pi_{\QQ}(a+\theta_1,b+\theta_2,\RR)$, if the polyradius $\RR$ is taken small enough.

\item If $f\sim^\PP \hat{f}$ on $\Pi_\PP(a,b;\RR)$, then $f\circ r_m\sim^{\PP\circ r_m} \hat{f}\circ r_m$ on $\Pi_{\PP\circ r_m}(a,b;\RR')$ and $f\circ b_\xi\sim^{\PP\circ b_\xi} \hat{f}\circ b_\xi$ on $\Pi_{\PP\circ b_\xi} (a,b;\RR')$ for any $m\geq 2$, $\xi\in\mathbb{P}_\C^1$ for some $\RR'$ is small enough.
\end{enumerate}
\end{nota}

\begin{nota}\label{Remark 5}We remark that the germ $\PP$ we have worked with may not depend on all $\xx$. To fix ideas assume $\xx=(\xx_1,\xx_2)\in\C^n\times \C^{d-n}$ and $\PP(\xx)=\PP(\xx_1)$. Then the variables $\xx_2$ are interpreted as regular parameters and instead of working in $E$ we work in the Banach space $\mathcal{O}_b(D_{\boldsymbol{\rho}},E)$, for some $\boldsymbol{\rho}\in(\R^+)^{d-n}$.
\end{nota}

\section{Summability in an analytic germ}\label{Summability in an analytic germ}

In this section we recall $\PP$-$s$--Gevrey asymptotic expansions and summability in a germ of analytic function. In particular we define $P$-$s$--Gevrey series and find a new characterization in Lemma \ref{Characterization of P-s-Gevrey} that allow us to easily prove basic properties of these series. We also include the key Lemmas \ref{Bounds for formal gevrey series} and \ref{P and PN} that will be used in the last section of the paper.

We say that {$\hat{f}\in\hat{\mathcal{O}}(E)$} is a \textit{$P$-$s$--Gevrey series}, $s\geq0$, if there is a $P$--asymptotic sequence $\{f_n\}_{n\in\N}\subset \mathcal{O}_b(D_r^d,E)$ for $\hat{f}$, such that $\|f_n(\xx)\|\leq CA^n n!^s$, for all $|\xx|\leq r$, $n\in\N$. In this case we say that $\{f_n\}_{n\in\N}$ is a \textit{$P$-$s$--asymptotic sequence} for $\hat{f}$. We will use the notation $\hat{\mathcal{O}}^{P,s}(E)$ for the set of $P$-$s$--Gevrey series. In the case $\PP(\xx)=\xx^\aa$ we will write $E[[\boldsymbol{x}]]_{s}^{\boldsymbol{\a}}=\hat{\mathcal{O}}^{\xx^\aa,s}(E)$ instead.
	
Given any injective linear form $\ell:\N^d\rightarrow \R^+$, $\hat{f}\in \hat{\mathcal{O}}^{P,s}(E)$ if and only if there is $\rho>0$ and a sequence $\{g_n\}_{n\in\N}\subset \mathcal{O}_b(D_\rho^d,E)$ with $J(g_n)\in\Delta_\ell(P,E)$, for all $n\in\N$ such that $\hat{T}_{P,\ell}\hat{f}=\sum_{n=0}^\infty g_n t^n$ is a $s-$Gevrey series in $t$, see \cite[Def./Prop. 7.5]{Sum wrt germs}. In fact, the restriction on the supports of the $g_n$ can be removed, as the following lemma shows.

\begin{lema}\label{Characterization of P-s-Gevrey}Let $\hat{f}\in\hat{\mathcal{O}}(E)$ be a series. Then $\hat{f}\in \hat{\mathcal{O}}^{P,s}(E)$ if and only if there are $r,C,A>0$ and a sequence  $\{f_n\}_{n\in\N}\in\mathcal{O}_b(D_r^d)$ such that $\hat{f}=\sum_{n=0}^\infty f_n P^n$ and  $\|f_n(\xx)\|\leq CA^n n!^s$, for all $|\xx|\leq r$, $n\in\N$.	
\end{lema}

\begin{proof}{If $\hat{f}\in \hat{\mathcal{O}}^{P,s}(E)$, the statement follows by \cite[Def./Prop. 7.5]{Sum wrt germs} as seen above}. Conversely, let us fix an injective linear form $\ell:\N^d\rightarrow \R^+$. Using Proposition \ref{Generalized Weierstrass Division} and Remark \ref{Remark 1}(2) we can find $\rho>0$ small enough such that for all $n\in\N$ we can write $$f_n(\xx)=\sum_{j=0}^\infty f_{n,j}(\xx) \PP^j(\xx),\quad |\xx|<\rho,\quad \text{where } f_{n,j}=R_{\PP,\ell}\circ Q_{\PP,\ell}^j(f_n)\text{ and } J(f_{n,j})\in\Delta_\ell(\PP,E).$$ In particular we see that $\hat{T}_{\PP.\ell}\hat{f}=\sum_{N=0}^\infty F_N t^N,$ where $F_N=\sum_{j=0}^N f_{j,N-j}.$ Since the operators $R_{\PP,\ell}$ and $Q_{\PP.\ell}$ are linear and continuous their operator norms $\|R_{\PP,\ell}\|, \|Q_{\PP.\ell}\|$ are finite and we obtain the bound \begin{align*}\|F_N(\xx)\|&\leq \sum_{j=0}^{N} \|R_{\PP,\ell}\| \|Q_{\PP.\ell}\|^{N-j} \sup_{|\yy|\leq \rho} \|f_j(\yy)\|\\	
&\leq \sum_{j=0}^{N} \|R_{\PP,\ell}\| \|Q_{\PP.\ell}\|^{N-j} CA^j j!^s=C\|R_{\PP,\ell}\| \|Q_{\PP.\ell}\|^N \sum_{j=0}^{N} \left(\frac{A}{\|Q_{\PP.\ell}\|}\right)^j j!^s,\end{align*} for $|\xx|<\rho$. Then it is clear that we can find constants $B,D>0$ such that $\|F_N(\xx)\|\leq DB^N N!^s$, for all $|\xx|\leq \rho$, $N\in\N$ and thus $\hat{f}\in \hat{\mathcal{O}}^{P,s}(E)$.
\end{proof}

\begin{coro}Let $P,Q\in\mathcal{O}\setminus\{0\}$ such that $P(\00)=Q(\00)=0$. The following assertions hold: \begin{enumerate}
		\item $\hat{\mathcal{O}}^{P,s}(E)$ is stable under sums and partial derivatives. If $E$ is a Banach algebra, then $\hat{\mathcal{O}}^{P,s}(E)$ is also stable under products.
		\item For any $N\in\N^+$, $\hat{\mathcal{O}}^{P^N,Ns}(E)=\hat{\mathcal{O}}^{P,s}(E)$.
		\item If $Q$ divides $P$, then 	
		$\hat{\mathcal{O}}^{P,s}(E)\subseteq \hat{\mathcal{O}}^{Q,s}(E)$.
	\end{enumerate}
\end{coro}

\begin{proof}Fix an injective linear form $\ell:\N^d\rightarrow \R^+$ and let $\hat{f}\in \hat{\mathcal{O}}^{P,s}(E)$. Then there is $\rho>0$ and a sequence $\{g_n\}_{n\in\N}\subset \mathcal{O}_b(D_\rho^d,E)$ with $J(g_n)\in\Delta_\ell(P,E)$, for all $n\in\N$ such that $\hat{f}=\sum_{n=0}^\infty g_n \PP^n$ and $\|g_n(\xx)\|\leq CA^n n!^s$, for all $|\xx|\leq \rho$, $n\in\N$. 
		
To prove (1) note that $\frac{\d\hat{f}}{\d x_j}=\sum_{n=0}^\infty \left(\frac{\d g_n}{\d x_j}+(n+1)g_{n+1}\frac{\d P}{\d x_j}\right) \PP^n$ and then $\frac{\d\hat{f}}{\d x_j}\in \hat{\mathcal{O}}^{P,s}(E)$ follows from Cauchy inequalities applied to $g_n$ and Lemma \ref{Characterization of P-s-Gevrey}.

For (2) we recall that the limit \begin{equation}\label{Limit}
\lim_{n\to\infty}\frac{(nk)!^{1/k}}{k^n n!}n^{\frac12-\frac1{2k}}=\frac{(2\pi k)^{\frac{1}{2k}}}{\sqrt{2\pi}},\quad \text{ for any integer }k\geq1,
\end{equation} allows to interchange, up to a geometric factor of $n$, the terms $(nk)!^{1/k}$ and $n!$. To see that $\hat{\mathcal{O}}^{P,s}(E)\subseteq \hat{\mathcal{O}}^{P^N,Ns}(E)$ let $\hat{f}\in \hat{\mathcal{O}}^{P,s}(E)$ as before. Then $\hat{f}=\sum_{n=0}^\infty h_n\PP^{Nn}$, $h_n=\sum_{j=0}^{N-1} g_{nN+j}\PP^j$. Using the limit (\ref{Limit}) we can find constants $B,D>0$ such that $\|h_n(x)\|\leq DB^n n!^{Ns}$, for all $|\xx|\leq \rho$, $n\in\N$ and thus $\hat{f}\in \hat{\mathcal{O}}^{P^N,Ns}(E)$. Conversely, assume $\hat{f}=\sum_{n=0}^\infty \overline{g}_n \PP^{nN}\in \hat{\mathcal{O}}^{P^N,Ns}(E)$ and $\|\overline{g}_n(\xx)\|\leq CA^{n}n!^{Ns}$, for all $|\xx|\leq \rho$, $n\in\N$. Then $\hat{f}=\sum_{m=0}^\infty \overline{h}_m \PP^{m}$, where $\overline{h}_m=\overline{g}_n$ if $m=Nn$ and $0$ otherwise. The limit (\ref{Limit}) implies once more the required bounds for the $\overline{h}_m$ and then $\hat{f}\in \hat{\mathcal{O}}^{P,s}(E)$.

Finally to prove (3) assume $P=Q\cdot R$, where $R\in\mathcal{O}_b(D^d_\rho,\C)$. Then $\hat{f}=\sum_{n=0}^\infty (g_nR^n)\QQ^n$ and $\|g_n(\xx)R^n(\xx)\|\leq CB^n n!^s$, $B=A\cdot \sup_{|\xx|\leq \rho} |R(\xx)|$. Thus $\hat{f}\in \hat{\mathcal{O}}^{Q,s}(E)$ by Lemma \ref{Characterization of P-s-Gevrey}.
\end{proof}

For the case $P(\xx)=\xx^\aa$, we see that $\hat{f}$ is a $\xx^\aa$-$s$--Gevrey series if for some $r>0$, $\hat{T}_{\boldsymbol{\a}} \hat{f}\in\mathcal{E}_r^{\boldsymbol{\a}}[[t]]$ and it is a $s$--Gevrey series in $t$, i.e., there are constants $C,A>0$ such that $\|f_{\aa,n}\|\leq CA^n n!^s$, for all $n\in\N$. This condition can be directly identified from the coefficients of $\hat{f}$. The proof is the same as in \cite[Lemma 3.1]{C}	but it is included here for sake of completeness. {It is worth mentioning that Lemma \ref{Bounds for formal gevrey series} (2) below is the crucial point to prove our main result, namely, Theorem \ref{Main Result}.}

\begin{lema}\label{Bounds for formal gevrey series}Assume the entries of $\aa,\aa'\in\N^d$ are not zero and let $\hat{f}=\sum f_{\boldsymbol{\beta}}\boldsymbol{x}^{\boldsymbol{\beta}}\in \hat{\mathcal{O}}(E)$ be a series. Then:\begin{enumerate}\item  $\hat{f}\in E[[\boldsymbol{x}]]_{s}^{\boldsymbol{\a}}$ if and only if there are constants $C,A>0$ satisfying $$\|f_{\boldsymbol{\beta}}\|\leq CA^{|\boldsymbol{\beta}|}\min_{1\leq j\leq d} \beta_j!^{s/\a_j},\quad \boldsymbol{\beta}\in\N^d.$$ In particular, we obtain again that  $E[[\xx]]_{Ns}^{N\aa}=E[[\xx]]_{s}^{\aa}$ for all $N\in\N^+$.
		
\item If $\hat{f}\in E[[\boldsymbol{x}]]_{s}^{\boldsymbol{\a}'}$, then there exist $r>0$ such that $\hat{T}_{\boldsymbol{\a}} \hat{f} \mid_{D^d_r}$ is a $\max_{1\leq j\leq d}\{\a_j/\a_j'\}s$--Gevrey series in some $\mathcal{E}^{\boldsymbol{\a}}_r$.\end{enumerate}\end{lema}

\begin{proof}$(1)$ Assume there are constants $B,D>0$ such that $\|f_{\aa,n}\|\leq DB^n n!^s$, for all $n\in\N$. Given $\boldsymbol{\gamma}\in\N^d$, let $n=\min_{1\leq j\leq d}\lfloor\gamma_j/\a_j\rfloor$, where $\lfloor\cdot\rfloor$  denotes the floor function. Thus $\boldsymbol{\gamma}=n\aa+\bb$ with $\beta_l<\a_l$ for some $l$. Then by Cauchy's inequalities we see that $$\|f_{\boldsymbol{\gamma}}\|=\|f_{n\boldsymbol{\a}+\boldsymbol{\beta}}\|=\left\|\frac{1}{\boldsymbol{\beta}!}\frac{\d^{\boldsymbol{\beta}}f_{\aa,n}}{\d \boldsymbol{x}^{\boldsymbol{\beta}}}(\boldsymbol{0})\right\|\leq\frac{DB^n}{r^{|\boldsymbol{\beta}|}}n!^s,$$ which yields one implication. The converse follows by the same argument as in (2) below.
 	
 	$(2)$ If $\|f_{\boldsymbol{\beta}}\|\leq CA^{|\boldsymbol{\beta}|}\min_{1\leq j\leq d}\{\beta_j!^{s/\a_j'}\}$, for all $\boldsymbol{\beta}\in\N^d$, we can directly estimate the growth of the $f_{\aa,n}$ by means of equation (\ref{para definir Tpq}): if $|\xx|<r$,  and $rA<1$ we obtain\begin{align*}
 	\|f_{\aa,n}(\boldsymbol{x})\|=\left\|\sum_{\aa\not\leq\gg} f_{n\aa+\gg} \xx^\gg\right\|&\leq \sum_{j=1}^{d}\sum_{\beta_j=0}^{\a_j-1} \sum_{\gg\in \N^d, \gamma_j=\b_j} CA^{n|\aa|+|\gg|} r^{|\gg|} \min_{1\leq l\leq d}\{(n\a_l+\gamma_l)!^{s/\a_l'}\}\\
 	&\leq \frac{CA^{n|\boldsymbol{\a}|}}{(1-rA)^{d-1}}\sum_{j=1}^{d}\sum_{\beta_j=0}^{\a_j-1}   (n\a_j+\beta_j)!^{s/\a_j'}(rA)^{\beta_j}.
 	\end{align*} If we write $s'=\max_{1\leq j\leq d} \{\a_j /\a_j'\} s$, using the limit (\ref{Limit}) we can find constants $B,D>0$ such that $$(n\a_j+\beta_j)!^{s/\a_j'}\leq (\a_j(n+1))!^{s/\a_j'}\leq DB^n(n+1)!^{s \a_j /\a_j'}\leq DB^n(n+1)!^{s'}\leq D(2^{s'}B)^n n!^{s'},$$ for all $n\in\N$. Then it is clear that we can find constants $K,M>0$ such that $\|f_{\aa,n}(\boldsymbol{x})\|\leq KM^n n!^{s'}$, for all $|\xx|<r$ and all $n\in\N$, as we wanted to show.
 \end{proof}

For convergent series, i.e., for $s=0$, we also see directly that $\hat{f}\in \mathcal{O}(E)$ if and only if $\hat{T}_{\aa}\hat{f}\in \mathcal{E}^{\aa}_r\{t\}$ for some $r>0$. For the general case we also have $\hat{f}\in \mathcal{O}(E)$ if and only if $\hat{T}_{\PP,\ell}\hat{f}\in \mathcal{E}^{\PP}_{\ell,r}\{t\}$ for some $r>0$. One implication is the content of Remark \ref{Remark 1}(2). The converse follows by simply replacing $t=\PP(\xx)$.

\begin{defi}\label{Def Pasym Gevrey} Let $\Pi_\PP=\Pi_\PP(a,b;\RR)$ be a $\PP$-sector, $f\in\mathcal{O}(\Pi_\PP,E)$ and $\hat{f}\in\mathcal{O}(E)$. We will say that $f$ has $\hat{f}$ as \textit{$\PP$-$s$--Gevrey asymptotic expansion on $\Pi_\PP$} if $f\sim^\PP \hat{f}$ on $\Pi_\PP$ and furthermore:
	
	\begin{enumerate}
	\item One of the sequences $\{f_n\}_{n\in\N}$ of Definition \ref{Def P-asym} satisfies $\|f_N(\xx)\|\leq KA^N N!^s$, for all $N\in\N$, $|\xx|<r$. In this case $\{f_n\}_{n\in\N}$ is called a \textit{$\PP$-$s$--asymptotic sequence for $\hat{f}$}.
	
	\item We can find constants $C,A>0$ such that $C_N=CA^N N!^s$, where $C_N$ is the constant in inequality (\ref{Def P-asym inq}).
	\end{enumerate}
This notion is independent of the choice of the $\PP$-$s$--asymptotic series. We will denote this by $f\sim^{\PP}_{s} \hat{f}$ on $\Pi_{\PP}$. If $\PP(\xx)=\xx^\aa$, we will write $f\sim^{\aa}_{s} \hat{f}$ on $\Pi_{\aa}$.
\end{defi}

\begin{nota}\label{Remark 3}We remark the following facts on the notion of $\PP$-$s$--Gevrey asymptotic expansions:\begin{enumerate}
\item Definition \ref{Def Pasym Gevrey} is independent of the choice of the $\PP$-$s$--asymptotic sequence for $\hat{f}$.

\item The analog of Theorem \ref{C P-asym in terms of Tl} holds in this setting:
Let $\ell:\N^d\rightarrow \R^+$ be an injective linear form. Then $f\sim_s^\PP \hat{f}$ on $\Pi_\PP=\Pi_\PP(a,b;\RR)$ if and only if there exists $\rho>0$ such that $\hat{T}_{\PP,\ell}\hat{f}\mid_{D_\rho^d}=\sum_{n=0}^\infty f_{\PP,\ell,n}  t^n\in \mathcal{E}^\PP_{\ell,\rho}[[t]]$ is a formal $s$--Gevrey series and one of the following two equivalent conditions holds:\begin{enumerate}
		\item For every $\Pi'_\PP\subset \Pi_\PP$, there exist $C,B>0$ such that for every $N\in\N$ \begin{equation}\label{Formula P-s-asym}
		\left\|f(\boldsymbol{x})-\sum_{n=0}^{N-1}f_{\PP,\ell,n}(\xx)\PP(\xx)^n \right\|
       \leq C  B^N N!^s|\PP(\xx)|^N,\quad \text{ on } \Pi_{\PP}'\cap D_\rho^d.
		\end{equation}
		
		\item The function $T_{\PP,\ell}f$ from Theorem \ref{Existence TlP} is defined on $V(a,b;\sigma)\times D_\rho^d\rightarrow\C$ for some positive $\sigma$ and satisfies $$T_{\PP,\ell} f\sim_s \hat{T}_{\PP,\ell}\hat{f}\mid_{D_\rho^d} \text{ as } V (a,b;\sigma)\ni t\rightarrow 0.$$
	\end{enumerate}

\item If $f\sim_s^\PP \hat{f}$ on $\Pi_\PP=\Pi_\PP(a,b;\RR)$ and $\hat f$ and $P$ are divisible by some $x_j$, $j=1,\dots,d$, then $x_j^{-1}f\sim_s^\PP x_j^{-1}\hat{f}$ on $\Pi_P$. Indeed, assume that (\ref{Formula P-s-asym}) holds. The hypotheses on divisibility imply that $x_j^{-1}f_{P,\ell,0}(\xx)$ is analytic at the origin. Thus we can divide (\ref{Formula P-s-asym}) by $x_j$ to obtain $$\left\|x_j^{-1}f(\xx)-x_j^{-1}f_{P,\ell,0}(\xx)-\sum_{n=1}^{N-1}f_{P,\ell,n}(\xx)Q(\xx)P(\xx)^{n-1}\right\|\leq CK A^N N! |P(\xx)|^{N-1},$$ where $Q(\xx)=x_j^{-1}P(\xx)$ and $K=\sup_{\xx\in \Pi_\PP'}|Q(\xx)|$.

\item In the monomial case, if $f\sim^{\aa}_{s} \hat{f}$ on $\Pi_{\aa}$ in the sense of Theorem \ref{C P-asym in terms of Tl} (1), then it follows from inequalities (\ref{Formula Pasym}) that $\hat{f}\in E[[\boldsymbol{x}]]^{\boldsymbol{\a}}_{s}$ \cite[Prop. 3.11, Remark 3.12]{Sum wrt germs}.

\item $\PP$-$s$--asymptotic expansions are stable under addition and partial derivatives and if $E$ is a Banach algebra, they are stable under products as well. They are stable under left composition with analytic functions as well.
\end{enumerate}
\end{nota}

We can compare asymptotic expansions in different powers of some analytic germ. On this matter we will use the following lemma, proof of which follows the same lines as in \cite[Prop. 3.5]{CM} for the particular case treated there.

\begin{lema}\label{P and PN} $f\sim_s^\PP\hat{f}$ on $\Pi_\PP(a,b;\RR)$ if and only if $f\sim_{Ms}^{\PP^M}\hat{f}$ on $\Pi_\PP(a/M,b/M;\RR)$ for any $M\in\N^+$.
\end{lema}

\begin{proof} Let us fix an injective linear form $\ell:\N^d\rightarrow\R^+$. Using Remark \ref{Remark 1} (1), we see that $$\hat{T}_{\PP^M,\ell}\hat{f}=\sum_{n=0}^\infty (f_{\PP,\ell,nM}+f_{\PP,\ell,nM+1}\PP+\cdots+f_{\PP,\ell,nM+M-1}\PP^{M-1})t^n,\quad \hat{f}\in \hat{\mathcal{O}}^\PP(E).$$

If $f\sim_s^\PP\hat{f}$ on $\Pi_\PP(a,b;\RR)$, we can conclude that	also $f\sim_{Ms}^{\PP^M}\hat{f}$ on $\Pi_\PP(a/M,b/M;\RR)$ using inequalities (\ref{Formula Pasym}) for the values $N=Mp$, $p\in\N$ and the limit (\ref{Limit}) for $k=M$ to adjust the constant $L_{pM}=LB^{pM}(pM)!^s$ to a constant of the form $L'_p=L'D^p p!^{Ms}$.

Conversely, if $f\sim_{Ms}^{\PP^M}\hat{f}$ on $\Pi_\PP(a/M,b/M;\RR)$, this implies that inequalities (\ref{Formula Pasym}) hold only for the values $N=pM$, $p\in\N$:$$\left\|f(\boldsymbol{x})-\sum_{n=0}^{Mp-1}f_{\PP,\ell,n}(\xx)\PP(\xx)^n \right\|\leq C'B^p p!^{Ms}|\PP(\xx)|^{Mp}.$$ We can apply Theorem \ref{Existence TlP} (2) with $K(u)=u^{Mp}$ to conclude that \begin{equation}\label{Inq TPl power}
\left\|T_{\PP,\ell}f(t,\xx)-\sum_{n=0}^{Mp-1} f_{\PP,\ell,n}(\xx)t^n \right\|\leq LC'B^p p!^{Ms}|t|^{Mp-1},
\end{equation} in the corresponding sector but with $|t|<\sigma$ and $|\xx|<\rho$, where $\sigma,\rho>0$ are small enough. Using (\ref{Inq TPl power}) for $p$ and $p+1$ we conclude that $$\|f_{\PP,\ell,Mp}(\xx)t+ f_{\PP,\ell,Mp+1}(\xx)t^2+\cdots+ f_{\PP,\ell,Mp+M-1}(\xx)t^{M}\|\leq LC'B^p p!^{Ms}+LC'B^{p+1}(p+1)!^{Ms}|t|^M,$$ in the same domain. Applying Lemma \ref{Vandermonde} with $K(u)=LC'B^p p!^{Ms}\left(1/u+B(p+1)^{Ms}u^{M-1}\right)$ we can conclude that $\hat{T}_{\PP,\ell}\hat{f}$ is indeed $s$--Gevrey in $t$, i.e., there are constants $C,A>0$ such that $\|f_{\PP,\ell,n}(\xx)\|\leq CA^n n!^s$, for all $|\xx|<\rho$ and $n\in\N$. Using again (\ref{Inq TPl power}) for $p$ and $p+1$ it is straightforward to check that $$\left\|T_{\PP,\ell}f(t,\xx)-\sum_{n=0}^{Mp-1} f_{\PP,\ell,n}(\xx)t^n \right\|\leq C''(A'')^p (pM)!^{s}|t|^{Mp},$$ for large enough constants $C'', A''>0$ independent of $p$. An application of Remark \ref{Lema polinomios} shows that $T_{\PP,\ell} f\sim_s \hat{T}_{\PP,\ell}\hat{f}$ in $V(a,b;\sigma)$ as we wanted to show.
\end{proof}

For Gevrey asymptotic expansions in one variable, we know that $f\sim_s 0$ on $S$ if and only if for every subsector $S'\subset S$ there are constants $C,A>0$ such that $\|f(t)\|\leq C\exp(-1/A|t|^{1/s}),$ $t\in S'$. Furthermore the cornerstone to define $k$--summability in one variable is \textit{Watson's lemma}: if $f\sim_s 0$ on $S(\theta,b-a;r)$ and $b-a> s\pi$, then $f\equiv 0$. Then it is natural to say for $\hat{f}\in \hat{\mathcal{O}}(E)$, $k>0$ and $\theta\in\R$ that: \begin{enumerate} \item The series $\hat{f}$ is \textit{$k$--summable on $S=S(\theta,b-a;r)$ with sum $f\in\mathcal{O}(S,E)$} if $b-a>\pi/k$ and $f\sim_{1/k} \hat{f}$ on $S$. We also say that $\hat{f}$ is \textit{$k$--summable in the direction $\theta$}. The corresponding space is denoted as $E\{x\}_{1/k,\theta}$.
	
	\item The series $\hat{f}$ is \textit{$k$--summable} if it is $k$--summable in all directions up to a finite number of them mod. $2\pi$ (the singular directions). The corresponding space is denoted as $E\{x\}_{1/k}$.\end{enumerate}

For proofs in Section \ref{Tauberian properties for summability}, we recall the following characterization of $k$--summability in terms of Borel-Laplace transformations.

\begin{prop}\label{laplace}
A series $\hat f(x)=\sum_{n=0}^\infty a_n x^n\in E[[ x]]_{1/k}$ is $k$--summable 
in a direction $\theta$ if and only if the following statements hold:
\begin{enumerate}
\item Its formal Borel transform $g(t)=\sum_{n=0}^\infty {a_n} \xi^n/{\Gamma(1+n/k)}$ is 
analytic in a neighborhood of the origin.
\item The function $g$ can be continued analytically in some infinite sector 
$S=V(\theta-\delta,\theta+\delta;\infty)$ containing the ray $\arg \xi=\theta$.
\item It has exponential growth there, i.e.,\ 
there are positive constants such that
$$
\|g(\xi)\|\leq C\cdot \exp\left({A}/{|\xi|^{k}}\right),\quad \xi\in S.
$$
\end{enumerate} Hence the Laplace integral
$f(x)=\int_{\arg \xi=\tilde \theta} e^{-(\xi/x)^k}
\ {g} (\xi) d(\xi/x)^k$ defining the
sum of $\hat f$ converges for $x$ in a certain sector
$V=V\left(\theta-\pi/2k-\tilde\delta/k,\theta+\pi/2k+\tilde\delta/k;r\right)$,
$0<\tilde\delta<\delta$, and suitably chosen $\tilde\theta$ close to $\theta$.
It satisfies $f\sim_{1/k}\hat f$ on $V$.
\end{prop}

For null $\PP$-$s$--Gevrey asymptotic expansions we have the two similar statements: First, $f\sim^\PP_{1/k} \hat{0}$ on $\Pi_\PP$ if and only if for every subsector $\Pi_\PP'\subset \Pi_\PP$ there are constants $C,A>0$ such that $$\|f(\xx)\|\leq C\exp(-1/A|\PP(\xx)|^{k}),\quad \xx\in\Pi_\PP'.$$ Second, we have a version of \textit{Watson's lemma}: If  $f\sim^\PP_{1/k} \hat{0}$ on $\Pi_\PP(a,b;\RR)$ and $b-a>\pi/k$, then $f\equiv 0$. These statements justify the following definition.

\begin{defi}\label{ddd}Let $\hat{f}\in \hat{\mathcal{O}}(E)$, $k>0$ and $\theta\in\R$ be a direction.
	\begin{enumerate}
		\item The series $\hat{f}$ is called \textit{$\PP$-$k$--summable on $S_{\PP}=S_{\PP}(\theta,b-a,\RR)$} with sum $f\in\mathcal{O}(S_{\PP},E)$ if $b-a>\pi/k$ and $f\sim_{1/k}^{\PP} \hat{f}$ on $S_{\PP}$. We also say that $\hat{f}$ is \textit{$\PP$-$k$--summable in the direction $\theta$}. The space of $\PP$-$k$--summable series in the direction $\theta$ will be denoted by $E\{\xx\}^{\PP}_{1/k,\theta}$.
		\item The series $\hat{f}$ is called \textit{$\PP$-$k$--summable}, if it is $\PP$-$k$--summable in all directions up to a finite number of them mod. $2\pi$ (the singular directions). The corresponding space is denoted as $E\{\xx\}^{\PP}_{1/k}$.
\end{enumerate} \noindent If $\PP(\xx)=\xx^\aa$, we will simply write $E\{\xx\}^{\aa}_{1/k,\theta}$ and $E\{\xx\}^{\aa}_{1/k}$, respectively.
 	
Note that both $E\{\boldsymbol{x}\}^{\PP}_{1/k,\theta}$ and $E\{\boldsymbol{x}\}^{\PP}_{1/k}$ are vector spaces stable by partial derivatives and they inherit naturally a structure of algebra if $E$ is a Banach algebra.
\end{defi}

\begin{nota}\label{Remark 4} We emphasize the following properties that will be used in the next section:	
\begin{enumerate}
\item If $\PP,\QQ\in \mathcal{O}\setminus\{0\}$, $\PP(\00)=\QQ(\00)=0$, are associated, then it follows from Remark \ref{Remark 2}(4) that $$E\{\xx\}^{\PP}_{1/k,\theta}=E\{\xx\}^{\QQ}_{1/k,\theta},\quad \text{ and }\quad E\{\xx\}^{\PP}_{1/k}=E\{\xx\}^{\QQ}_{1/k}.$$

\item Lemma \ref{P and PN} implies that  $$E\{\xx\}^{\PP^N}_{N/k,N\theta}=E\{\xx\}^{\PP}_{1/k,\theta},\quad E\{\xx\}^{\PP^N}_{N/k}=E\{\xx\}^{\PP}_{1/k},\quad N\in\N^+.$$	

\item By Remark \ref{Remark 2}(5) we see that if $\hat{f}\in E\{\boldsymbol{x}\}^{\PP}_{1/k,\theta}$, then  $\hat{f}\circ r_m\in E\{\boldsymbol{x}\}^{\PP\circ r_m}_{1/k,\theta}$ and $\hat{f}\circ b_\xi\in E\{\boldsymbol{x}\}^{\PP\circ b_\xi}_{1/k,\theta}$, for all $m\geq 2$ and all $\xi\in\mathbb{P}_\C^1$. The converse is even more interesting and also true, although we will not use it in this paper.
\end{enumerate}
\end{nota}

\section{Tauberian properties for $\PP$-$k$--summability}\label{Tauberian properties for summability}

In one variable, we have the following classical statements providing tauberian properties for $k$--summability which we will generalize for $k$--summability in an analytic germ.

\begin{teor}\label{Tauberian classical} The followings statements are true for $0<k<k'$ and $0<k_0, k_1,\dots,k_n$:\begin{enumerate}\item If $\hat{f}\in E\{t\}_{1/k}$ has no singular directions, then it is convergent.

\item $E[[x]]_{1/k'}\cap E\{t\}_{1/k}=E\{t\}_{1/k'}\cap E\{t\}_{1/k}=E\{t\}.$

\item Consider $\hat{f}_j\in E\{t\}_{1/k_j}$ for $j=1,\dots,n$ and assume that $0<k_1<\cdots<k_n$. Then  $\hat{f}_1+\cdots+\hat{f}_n=0$ implies that $\hat f_j\in E\{t\}$, for all $j=1,\dots,n$.
	\end{enumerate}
\end{teor}

We will use that, by Remark \ref{Remark 3}(2), a series $\hat{f}$ is $\PP$-$k$--summable in some direction $\theta$ if and only if 
there exist $r=r_\theta>0$ such that $\hat{T}_{\PP,\ell}\hat{f} \mid_{D^d_{r_\theta}}$ is $k$--summable in direction $\theta$ in $\mathcal{E}^\PP_{\ell,r_\theta}$  in the classical sense. Unfortunately, $r_\theta$ might tend to 0 when $\theta$ tends to a singular direction. Therefore, $\PP$-$k$--summability of a series $\hat{f}$ does not imply that $\hat{T}_{\PP,\ell}\hat{f} \mid_{D^d_r}$ is $k$--summable in $\mathcal{E}^\PP_{\ell,r}$ for some fixed $r>0$. For a counterexample, see \cite{Monomial summ}, Section 6. Nevertheless, we have

\begin{teor}\label{no sing directions then convergence for monomials}If $\hat{f}\in E\{\boldsymbol{x}\}^{\PP}_{1/k}$ has no singular directions, then $\hat{f}$ is convergent.
 \end{teor}

\begin{proof} We follow a classical proof of Theorem \ref{Tauberian classical}(1). First, choose an injective linear form $\ell:\N^d\to\R^+$. Let us write $\hat T_{P,\ell}\hat{f}(t)=\sum_{n=0}^\infty f_{P,\ell,n} t^n$ where $f_{P,\ell,n}\in \mathcal{E}_{\ell,R}^P$ with some $R>0$. If $\hat{f}$ is $P$-$k$--summable in all directions, 
then, by Remark \ref{Remark 3}(2), for all directions $\theta\in[0,2\pi]$,
there exists some $0<r_\theta\leq R$ such that  $\hat T_{P,\ell}\hat{f}(t)\mid_{D^d_{r_\theta}}$ is $k$--summable in direction $\theta$. By Proposition \ref{laplace}, this means that for every $\theta\in[0,2\pi]$, there exist
$\rho_\theta,\delta_\theta>0$ such that the $k$--Borel transform $g_\theta(\xx,\xi)=\sum_{n=0}^\infty \frac{f_{P,\ell,n}(\xx)}{\Gamma(1+n/k)}\xi^n$ is convergent and defines by analytic continuation a holomorphic function $g_\theta:D_{\rho_\theta}^{d}\times S(\theta,2\delta_\theta)\to \C$. Furthermore $g_\theta$ has exponential growth of order $k$, i.e., there exist 
$A_\theta,K_\theta>0$ such that 
$$\|g_{\theta}(\xx,\xi)\|\leq K_\theta\exp\left(A_\theta|\xi|^k\right),\quad \mbox{ for all }\xx, \xi \mbox{ in the domain.}$$
Now the open sets $(\theta-\delta_\theta,\theta+\delta_\theta)$, $\theta\in[0,2\pi]$, 
form an open covering of
the compact interval $[0,2\pi]$. Therefore there is a finite sub-covering, i.e.,\ 
a positive integer $N$ and $0\leq \theta_1<\dots <\theta_N\leq 2\pi$ such that
$[0,2\pi]\subset\cup_{j=1}^N [\theta_j-\delta_{\theta_j},\theta_j+\delta_{\theta_j}]$.
This means that the sectors $S(\theta_j,2\delta_{\theta_j})$ cover the punctured complex plane.
As every $g_{\theta_j}$ is an analytic continuation of the same germ at $\xi=0$, they can
be combined to 
a holomorphic function $g:D_{\rho}^d\times \C\to\C$, $\rho=\min_{1\leq j\leq N} \rho_{\delta_j}$, 
of exponential growth $\|g(\xx,\xi)\|\leq K \exp({A|\xi|^k})$, for all $\xi\in\C$, $|\xx|<\rho$, with the 
constants $K=\max_{1\leq j\leq N} K_{\theta_j}$ and $A=\max_{1\leq j\leq N} A_{\theta_j}$. 

It is well known that this implies the convergence of $\hat T_{P,\ell}\hat f$ and hence the 
convergence of $\hat f$. Indeed, Cauchy's inequalities on a disk of radius $(\frac n{Ak})^{1/k}$ show that
$$\|f_{P,\ell,n}(\xx)\|\leq K(Ae)^{n/k}\frac{\Gamma(1+n/k)}{(n/k)^{n/k}},\quad \mbox{ for all } |\xx|<\rho,\ n\in\N.$$ 
An application of Stirling's formula allows  to conclude that $\hat{f}$ is convergent.
\end{proof}

\begin{lema}\label{tauberian general case}Let $\aa,\aa'\in\N^d\setminus\{\00\}$ and $k, k'>0$. The following statements hold:
	
	\begin{enumerate}
		\item {If $\hat{f}\in E\{\boldsymbol{x}\}_{1/k}^{\boldsymbol{\a}}$ and $\hat{T}_\aa \hat{f}$ is an $s$--Gevrey series with some $s<1/k$, then $\hat{f}$ is convergent. In particular, if the entries of $\aa$ and $\aa'$ are not zero and $\max_{1\leq j\leq d}\{\a_j/\a_j'\}< k'/k$, then $E\{\boldsymbol{x}\}_{1/k}^{\boldsymbol{\a}}\cap  E[[\boldsymbol{x}]]_{1/k'}^{\boldsymbol{\a}'}=E\{\boldsymbol{x}\}$.}
		
		\item $E\{\boldsymbol{x}\}_{1/k}^{\boldsymbol{\a}}\cap E\{\boldsymbol{x}\}_{1/k'}^{\boldsymbol{\a}'}=E\{\boldsymbol{x}\}$, except in the case $k\aa=k'\aa'$ where $E\{\boldsymbol{x}\}_{1/k}^{\boldsymbol{\a}}=E\{\boldsymbol{x}\}_{1/k'}^{\boldsymbol{\a}'}$.
	\end{enumerate}
\end{lema}

\begin{proof} (1) The second statement is indeed a consequence of the first: If $\hat{f}\in E[[\xx]]_{1/k'}^{\aa'}$, then by Lemma \ref{Bounds for formal gevrey series}(2), $\hat{T}_{\boldsymbol{\a}}\hat{f}$ is a $\max_{1\leq j\leq d}\{\a_j/\a_j'\}/k'$--Gevrey series in some $\mathcal{E}_r^{\boldsymbol{\a}}$ and the first statement applies. 

For the proof of the first statement, we follow the proof of Theorem 3.8.2 in \cite{RS89}.
Let us write $\hat T_\a\hat f(t)=\sum_{n=0}^\infty f_{n}t^n$, with $f_n\in \mathcal{E}_r^{\boldsymbol{\a}}$ and 
use the
$k$--Borel transform $g$ of $\hat T_\a\hat f$ in the form $g(\xx,\xi)=\sum_{n=0}^\infty \frac{f_{n}(\xx)}{\Gamma(1+n/k)}\xi^n$.
Since $\hat T_\a\hat f$ is $s$--Gevrey with some $s<1/k$, as seen above Lemma \ref{Characterization of P-s-Gevrey}, we find constants $K,A>0$ such that $$\frac{\|f_{n}(\xx)\|}{\Gamma(1+n/k)}\leq K A^n n!^{-1/\mu},\quad \text{for all } |\xx|<r, n\in\N,\quad  1/\mu:=1/k-s.$$ As is well known, this implies that $g$ is not only convergent, but defines a holomorphic function on $D_r^d\times\C$ having
exponential growth of order at most $\mu$ with respect to $\xi$, i.e., there are $L,B>0$ such that
\begin{equation*}\label{expomu}\|g(\xx,\xi)\|\leq L \exp(B|\xi|^\mu),\quad \mbox{ for all }|\xx|<r, \xi\in\C.\end{equation*}
Now we claim that $\hat f$ is $\xx^\a$-$k$--summable in all directions and hence convergent by Theorem 5.2
which proves statement (1).
We give a proof by contradiction. Assume that $\theta$ is a singular direction of $\hat f$.
We choose a positive $\delta<\frac\pi{2\mu}$ such that, again by Remark \ref{Remark 3}(2), a certain restriction of $\hat T_\a\hat f$ is $k$--summable in the directions $\theta_-=\theta-\delta$
and $\theta_+=\theta+\delta$. 
By Proposition \ref{laplace}, there exist $0<\rho<r$ and $M,C>0$ such that the $k$--Borel transform $g$ of $\hat T_\a\hat f$ satisfies
\begin{equation*}\label{expok}\|g(\xx,\xi)\|\leq M\exp(C|\xi|^k),\quad \mbox{ for all }|\xx|<\rho, \arg(\xi)\in\{\theta_-,\theta_+\}.\end{equation*}
We want to use the Phragm\'{e}n-Lindel\"{o}f principle on the sector $\arg(t)\in [\theta_-,\theta_+]$. We apply it to the function
$h(\xx,\xi)=g(\xx,\xi)\exp\left(-D(\xi e^{-{i\theta}})^k\right)$ where $D$ is chosen such that 
$\left|\exp\left(D(\tau e^{{i\delta}})^k\right)\right|=\exp(C\tau ^k)$ for positive $\tau $. This means that $D\cos(k\delta)=C$.
Therefore $h$ is bounded on the rays $\arg(\xi)\in\{\theta_-,\theta_+\}$, satisfies 
$\|h(\xx,\xi)\|\leq L \exp(B|\xi|^\mu),\mbox{ for all }|\xx|<\rho$, $\arg(\xi)\in[\theta_-,\theta_+]$. Finally, the opening of the sector is
smaller than $\pi/\mu$. Hence the Phragm\'{e}n-Lindel\"{o}f principle yields that $h$ is bounded on the full sector. Thus we can find constants $\tilde M,\tilde C>0$ such that $$\|g(\xx,\xi)\|\leq \tilde M\exp(\tilde C|\xi|^k),\quad \mbox{ for all }|\xx|<\rho, \arg(\xi)\in[\theta_-,\theta_+].$$ In particular,
$\hat T_\a\hat f\mid_{D_\rho^d}$ is also $k$--summable in direction $\theta$ by Proposition \ref{laplace}.  Therefore, by Remark \ref{Remark 3}(2), $\hat f$ is $\xx^\a$-$k$--summable in direction $\theta$ contradicting the assumption.

(2) If $k\aa=k'\aa'$, then $k/k'=p/q$ for some $p,q\in\N^+$, $(p,q)=1$ and thus $p\aa=q\aa'$. Then using Remark \ref{Remark 4}(2) we obtain $$E\{\xx\}_{1/k}^{\aa}=E\{\xx\}_{p/k}^{p\aa}=E\{\xx\}_{q/k'}^{q\aa'}=E\{\xx\}_{1/k'}^{\aa'}.$$ If $k\aa\neq k'\aa'$ we can use Lemma \ref{Ordered monomials} to find a monomial blow-up $\pi:\C^d\rightarrow\C^d$ such that $(\aa_1,1/k)=\pi^\ast(\aa,1/k)$ and $(\aa_2,1/k')=\pi^\ast(\aa',1/k')$ are comparable and the new monomials have no nonzero entries, i.e., we are in the situation of item (1) due to Remark \ref{Note on <}. If  $\hat{f}\in E\{\xx\}_{1/k}^{\aa}\cap  E\{\xx\}_{1/k'}^{\aa'}$, then by Remark \ref{Remark 4}(3) $\hat{f}\circ\pi \in E\{\xx\}_{1/k}^{\aa_1}\cap  E\{\xx\}_{1/k'}^{\aa_2}=E\{\xx\}$ and by Lemma \ref{Convergence is preserved by blow-ups} also $\hat{f}\in E\{\xx\}$. 
\end{proof}

\begin{nota} Theorem \ref{no sing directions then convergence for monomials} and Lemma \ref{tauberian general case} were obtained in \cite{C,CM}. Although the statements are correct, the proofs given there were based on the false statement discussed above Theorem 5.2. This is repaired here.
\end{nota}

Recall from Section \ref{Remarks on Monomialization} that for $P_0,P_1\in\mathcal{O}\setminus\{0\}$, $P_0(\00)=P_1(\00)=0$ and $k_0, k_1>0$ the couples $(P_0,1/k_0)\sim (P_1,1/k_1)$ if we can find $p_0, p_1\in\N^+$ and $U\in\mathcal{O}^\ast$ such that $p_0/k_0=p_1/k_1$, and $\PP_1^{p_1}= U\cdot \PP_0^{p_0}.$ It follows from Remark \ref{Remark 4}(1) and (2) that if $$(P_0,1/k_0)\sim (P_1,1/k_1),\quad \text{ then }\quad E\{\xx\}^{\PP_0}_{1/k_0}=E\{\xx\}^{\PP_1}_{1/k_1}.$$ The converse is also true and in fact, we can generalize Theorem \ref{Tauberian classical} (2) and (3) for $P$-$k$--summability as follows.

\begin{teor}\label{Main Result} Let $\PP_j\in\mathcal{O}\setminus\{0\}$, $\PP_j(\00)=0$, $k_j>0$ for $j=1,\dots,n$. For each $j=1,\dots,n$ consider a series $\hat{f}_j\in E\{\xx\}_{1/k_j}^{\PP_j}$. If the couples $(P_j,1/k_j)$, $j=1,\dots,n$ are pairwise not equivalent and  $\hat{f_1}+\cdots+\hat{f}_n=0$, then $\hat{f}_j\in E\{\xx\}$, for  all $j=1\dots,n$.
	
In particular, $E\{\xx\}^{\PP_0}_{1/k_0}=E\{\xx\}^{\PP_1}_{1/k_1}$ if and only if $(P_0,1/k_0)\sim (P_1,1/k_1)$.\end{teor}

\begin{proof} We proceed by induction on $n$. If $n=1$, there is nothing to prove. Assume now that the statement is true for some $n-1\geq1$. To show it holds for $n$ we proceed by induction on $N=h(\prod_{j=1}^n\PP_j)$, where $h$ is the function in Lemma \ref{Normalization}. For $N=0$ we can assume $P_j(\xx)=\xx^{\aa_j}$ for all $j$. Here the hypothesis on the couples is equivalent to the fact that $(\aa_1,1/k_1),\dots,(\aa_n,1/k_n)\in \Lambda_d$ are all distinct, more precisely, the products $k_j\aa_j$ are pairwise different. By Lemma \ref{Ordered monomials} there is a monomial blow-up $\pi:\C^d\rightarrow\C^d$ such that the elements $(\aa_j',1/k_j)=\pi^\ast(\aa_j,1/k_j)$, $j=1,\dots,n$ are totally ordered with respect to  $\prec$. Re-indexing if necessary we assume that $(\aa_1',1/k_1)\prec(\aa_2',1/k_2)\prec\cdots\prec (\aa_n',1/k_n).$ Using $\hat f_1=-\hat f_2-\cdots-\hat f_n$ and Lemma \ref{Bounds for formal gevrey series}(2) we conclude that $\hat{T}_{\aa_1'}(\hat{f_1}\circ\pi)$ is Gevrey of some value less than $1/k_1$. As $\hat{f}_1 \circ \pi$ is $\xx^{\aa_1'}$-$k_1$--summable, Lemma 5.3 (1) applies and yields the convergence of $\hat{f}_1 \circ\pi$. By Lemma \ref{Convergence is preserved by blow-ups}, $\hat{f}_1$ is convergent. We can apply the induction hypothesis to the $n-1$ series $\hat{f}_1+\hat{f_2}, \hat{f_3},\dots, \hat{f_n}$ to obtain the statement for the present $n$ and $N=0$.
	
Now suppose the statement is true whenever we have fewer than $n$ series or if $h(\prod_{j=1}^nP_j)<N$ for some $N>0$. By Lemma \ref{Normalization} there exists a diffeomorphism $D\in \text{Diff}(\C^d,\00)$ such that $h(\prod_{j=1}^n\PP_j\circ D\circ r_m)<N$ for some $m\geq 2$ or $h(\prod_{j=1}^n\PP_j\circ D\circ b_\xi)<N$, for all $\xi\in\mathbb{P}^1_\C$. Let us write $\QQ_j=\PP_j\circ D$ and $\hat{g}_j=\hat{f}_j\circ D$, $j=1,\dots,n$. We consider the two possibilities: \begin{enumerate}
	
\item $h(\prod_{j=1}^n\QQ_j\circ r_m)<N$ for some $m\geq 2$. By Remark \ref{Remark 4}(3) we have  $\hat{g}_j\circ r_m\in E\{\xx\}_{1/k_j}^{\QQ_j\circ r_m}$, and by Lemma \ref{Equivalence relation and blow ups} the {couples} $(Q_j\circ r_m,1/k_j)$ are pairwise not equivalent. By the induction hypothesis,  $\hat{g}_j\circ r_m\in E\{\xx\}$, and so, by Lemma \ref{Convergence is preserved by blow-ups}(2), $\hat{g}_j\in E \{ \xx\}$ and hence  $\hat{f}_j\in E \{ \xx\}$ for all $j=1,\dots,n$.

\item Assume now that $h(\prod_{j=1}^n\QQ_j\circ b_\xi)<N$, for all $\xi\in\mathbb{P}^1_\C$. By Lemma \ref{Equivalence relation and blow ups} there exists $\xi_0 \in \mathbb{P}^1_\C$ such that $\left( Q_1\circ b_{\xi_0},1/k_1 \right)$, $\left( Q_2\circ b_{\xi_0}, 1/k_2 \right)$ are not equivalent.
We group the germs $\hat{g}_j \circ b_{\xi_0}$ with respect to equivalence of the couples $\left( Q_j\circ b_{\xi_{0}},1/k_j \right)$, obtaining a partition $I_1\cup \cdots\cup I_{n'}$ of $\{ 1,2,\ldots ,n\}$ with $n'>1$ because at least two of the couples are not equivalent. Observe that for each $i$ the spaces $E\{\xx\}_{1/k_j}^{\QQ_j\circ b_{\xi_0}}$, $j\in I_i$ are identical. For each $i$, we fix some $j(i)\in I_i$. Now the $n'$ couples $(\QQ_{j(i)}\circ b_{\xi_0},1/k_{j(i)})$, $i=1,\dots,n'$ are pairwise not equivalent by construction. Denoting by $\hat{h}_{i}=\sum_{j\in I_i} \hat{g}_{j}\circ b_{\xi_{0}}$, we have $\hat{h}_{1}+\cdots + \hat{h}_{n'}=0$ and $\hat h_i\in E\{\xx\}_{1/k_{j(i)}}^{\QQ_{j(i)}\circ b_{\xi_0}}$, $i=1,\dots,n'$.
Since  $h(\prod_{i=1}^{n'}\QQ_{j(i)}\circ b_{\xi_0})\leq h(\prod_{j=1}^n\QQ_j\circ b_{\xi_0})<N$,
the induction hypothesis on $N$ yields that $\hat{h}_{i}\in E \{ \xx \}$ for all $i$. By Lemma  \ref{Convergence is preserved by blow-ups}(2), we conclude that $G_i:=\sum_{j\in I_i} \hat{g}_j\in E \{ \xx\}$ for each $i$. If some $I_i$ contains more than one element, we fix such an $i$ temporarily and change some $\hat{g}_j$ to $\hat{g}_j-G_i$. Then we can apply the induction hypothesis on $n$ because $|I_i|<n$ and obtain that all $\hat g_j$, $j\in I_i$ converge. As $i$ is arbitrary here, we conclude that $\hat g_j\in E\{\xx\}$ and hence $\hat f_j\in E\{\xx\}$ for all $j=1,\dots,n$.\end{enumerate}
Finally the principle of induction allows us to conclude the proof.


\end{proof}

\begin{nota}\label{last} (1) Theorem \ref{Main Result} in particular implies that a divergent series cannot be summable with respect to two germs if they depend on different variables. This observation has been applied in \cite{CM} to obtain convergence of formal solutions of Pfaffian systems with normal crossings and giving an alternative proof of G\'{e}rard-Sibuya theorem.

(2) As a consequence of Theorem \ref{Main Result}, if some formal series $\hat f$ can be written in two ways $$\hat f= \hat g_1+\cdots+ \hat g_n=\hat h_1+\cdots+ \hat h_n,$$ where $\hat g_j,\hat h_j\in E\{\xx\}_{1/k_j}^{\PP_{j}}$ and the couples $(\PP_j,1/k_j)$ are pairwise not equivalent, then these decompositions are essentially the same in the sense that all differences $\hat{g}_j-\hat{h}_j$ are analytic. This is a first step towards a definition of multisummability with respect to analytic germs.
\end{nota}

Unfortunately {\em sums} of $P_j$-$k_j$--summable series are not sufficient to define a multisummability compatible, e.g.,\ with products. We give an example of a product of summable series that cannot be a sum of Gevrey series.

\begin{eje} Consider the series $\hat{f}(t)=\sum_{n\geq 0} n! t^{n}$ which is known to be $1$--summable and the product $$\hat{F}(x_1,x_2)= \hat{f}(x_1)\cdot \hat{f}(x_2) = \sum_{k,l\geq 0} k! l! x_1^k x_2^l.$$ Assume that $\hat{F}(x_1,x_2)= \hat{g}_1 (x_1,x_2)+ \hat{g}_2 (x_1,x_2)$, where $\hat{g}_j$ is $x_j$-$1$--Gevrey, $j=1,2$. Write
$\hat{g}_j (x_1,x_2)=\sum_{k} g_{1k}(x_2)x_1^k =\sum_{k,l\geq 0} g_{1kl}x_1^kx_2^l$. There are constants $C,A>0$ such that $|g_{1k}(x_2)|\leq CA^k k!$, for small $x_2$ and for all $k\in\N$. By Cauchy's inequalities, there is another constant $B>0$ such that  $|g_{1kl}|\leq CA^k B^l k!$, for all $k,l\in\N$. Analogously, if we write $\hat{g}_2(x_1,x_2)=\sum_{k,l\geq 0} g_{2kl}x_1^kx_2^l$, we find that $|g_{2kl}|\leq CA^k B^l l!$, for all $k,l\in\N$. Here, without loss of generality, we use the same constants. If we consider the coefficients of $x_1^kx_2^k$, we would have that $k!^2\leq 2 C (AB)^k\,k!$, for all $k\in\N$, which is impossible.

\end{eje}

The example shows that it is desirable to have at least an analog of Remark \ref {last}(2) for products of summable series. As a corollary of Theorem \ref{Main Result} we prove here the following weaker statement for the case $E=\C$.

\begin{coro}\label{weak Tauber product} Let $\PP_j\in\mathcal{O}\setminus\{0\}$, $\PP_j(\00)=0$, $k_j>0$ for $j=1,\dots,n$. For each $j=1,\dots,n$ consider a series $\hat{f}_j\in \C\{\xx\}_{1/k_j}^{\PP_j}$. If $$\hat{f_0}=\hat{f_1}\cdots\hat{f}_n\in \C\{\xx\},$$ and the couples $(\PP_j,1/k_j)$ are pairwise not equivalent, then $\hat f_j$ are convergent for all $j=1,\dots,n$.\end{coro}
	
\begin{proof} There is a sequence $\phi$ of monomial blow ups, ramifications and right compositions with analytic diffeomorphisms such that $\hat{f}_0\circ \phi$ and $\PP_1\circ \phi,\cdots,\PP_n\circ \phi$ are products of a monomial times a unit in $\C\{\xx\}$. We can augment $\phi$ by a monomial blow up such that additionally, all the monomial factors of the $\PP_j\circ \phi$ contain every variable $x_1,\dots,x_d$. Now $\hat{f_0}\circ \phi=(\hat{f}_1\circ \phi)\cdots(\hat{f}_n\circ \phi)$ and hence the $\hat f_j\circ \phi$ are also products of monomials and units. Thus if we write $(\hat f_j\circ \phi)(\xx)=\xx^{\aa_j}\hat U_j(\xx)$, $\hat{U}_j\in\widehat{\mathcal{O}}^\ast$ a unit, $j=0,1,\dots,n$ we must have $\aa_0=\aa_1+\cdots+\aa_n$. By Remark \ref{Remark 3} (3), we can divide by the factors of the monomials and obtain $\hat U_0=\hat U_1\cdots\hat U_n$, where $\hat U_0$ is convergent and the $\hat U_j$ are $(\PP_j\circ \phi)$-$k_j$--summable. Taking the logarithm, we arrive essentially at the situation of Theorem \ref{Main Result}.
\end{proof}

\begin{eje}There are singularly perturbed differential equations with a formal solution being not $\PP$-$k$--summable for any $\PP$ or $k>0$. We provide an example based on the one given by J.P.\ Ramis and Y.\ Sibuya in \cite{RS89} for the case of one variable.
	
Consider Euler's equation $t^2y'+y=t$ and its formal solution $\hat{E}(t):=\sum_{n=0}^\infty (-1)^n n! t^{n+1}$. For any germ $\PP\in\mathcal{O}\setminus\{0\}$ such that $\PP(\00)=0$, the series $\hat y=\hat{E}(\PP)$ is $\PP$-$1$--summable and it satisfies the system of equations $\PP^2\frac{\d y}{\d x_j}+\frac{\d \PP}{\d x_j}y=\frac{\d \PP}{\d x_j}\PP$,  $j=1,\dots,d.$ 
			
Let us consider the skew-ring of differential operators $\mathbb{C}(\{ \xx\})[\partial_1,\ldots, \partial_d]$, $\partial_j=\partial_{x_j}$,with product satisfying $[\partial_i,\partial_j]=0$ and $\partial_j\cdot f= f\partial_j+ \partial_jf$, for every germ $f\in \mathbb{C}(\{ \xx\})$. For any $\PP\in\mathcal{O}\setminus\{0\}$, let $L_{P,j}:= P^2 \partial_{j} + \partial_j P$, that verifies $L_{P,j}(\hat{E}(P))= \partial_j P\cdot P$. 
	
To construct an operator having $\hat{E}(P)+ \hat{E}(Q)$ as solution, for fixed $P,Q\in\mathcal{O}\setminus\{0\}$, we can look for a right least common multiple  of $L_{P,j}$ and $L_{Q,j}$: it must be an operator $L_j$ such that $L_j= M_{P,j} L_{P,j}= M_{Q,j}L_{Q,j}$, for some $M_{P,j}$, $M_{Q,j}$. Indeed, if $L_{j}= A_j \partial_j^{2}+B_j \partial_j +C_j$, performing division to the right by $L_{P,j}$, we obtain $$L_j=M_{P,j}L_{P,j}+R_{P,j}, \quad  M_{P,j}=\frac{A_j}{\PP^2} \d_{j}+\frac{1}{\PP^2}\left(B_j-\frac{A_j}{\PP^2}(2\PP+1)\d_j P\right),$$ and $R_{P,j}=C_j-\frac{A_j}{\PP^2}\d^2_j \PP-\frac{1}{\PP^2}\d_j \PP\left(B_j-\frac{A_j}{\PP^2}(2\PP+1)\d_j\PP\right)$. If we require that $R_{P,j}=0$ and $R_{Q,j}=0$, then the equation $R_{P,j}=R_{Q,j}$ determines $A_j/B_j$. Thus,  we can choose 
	\begin{align*}
	A_j  = &  P^{2} Q^{2} \left(   Q^{2}\partial_j P    -  P^{2}\partial_j Q  \right) , \\
	B_j  = & Q^{4}\left(\left( 2P+1 \right)\left( 
	\partial_j P\right) ^{2}- P^{2}\partial^2_j P 
	\right) -  P^4\left(\left( 2Q+1\right)  \left( \partial_j Q \right) ^{2}-\partial^2_j Q\right) Q^{2},
	\end{align*} and $C_j$ determined by the equalities $R_{P,j}=0$ or $R_{Q,j}=0$. It follows that  $\hat{E}(\PP)+\hat{E}(\QQ)$ is a formal solution of the system \begin{equation}\label{Example Sol PDEs} L_j(y)=M_{\PP,j}\left(\PP \d_j\PP\right)+M_{\QQ,j}\left(\QQ \d_j\QQ \right),\quad j=1,\dots,d.
	\end{equation}  Note that $A_j=0$, for all $j=1,\dots,d$ if and only if $\QQ=U\PP$, where $U=1+c\QQ\in\mathcal{O}^\ast$, $c\in\C$. In this case we don't obtain a new equation since $L_j=Q^2L_{P,j}=P^2L_{Q,j}$, for all $j$. In fact, the solution $\hat{E}(\PP)+\hat{E}(\QQ)$ is $\PP$-$1$--summable. If this is not the case, we can use Theorem \ref{Main Result} to conclude that  $\hat{E}(\PP)+\hat{E}(\QQ)$ is not $\PP$-$k$--summable, for any $\PP, k$, but it is still a formal solution of the system (\ref{Example Sol PDEs}). Finally, if $P$ and $Q$ are polynomials, so is $L_j (\hat{E}(\PP)+\hat{E}(\QQ))$, and $\hat{E}(\PP)+\hat{E}(Q)$ is a solution of the polynomial differential equation $\partial_j^{N} L_j (y)=0$,  for an appropriate $N\in \N$.



We refer the reader to Examples 8.1 and 8.2 in \cite{Sum wrt germs} for a singular ordinary and a partial differential equation with $\PP$-$1$--summable formal solutions, respectively, where $\PP$ is a polynomial in two variables with certain conditions. We note that due to Theorem \ref{Main Result}, $\PP$-$1$--summability is essentially the only $\QQ$-$k$--summability method applicable to these formal solutions. In particular, monomial summability is not sufficient to sum these power series.

\end{eje}

\bibliographystyle{plaindin_esp}

\end{document}